\documentclass{amsart}
\usepackage[linesnumbered,lined,boxed,commentsnumbered]{algorithm2e}
\usepackage{amsthm}
\usepackage{amsmath}
\usepackage{amssymb}
\usepackage{bm}
\usepackage{graphicx}
\usepackage[shortlabels]{enumitem}
\usepackage{color}
\usepackage{hyperref}
\usepackage[square,numbers]{natbib}
\hypersetup{
	colorlinks,
	citecolor=black,
	filecolor=black,
	linkcolor=black,
	urlcolor=black
}

\newtheorem{theorem}{Theorem}[section]

\newtheorem{proposition}[theorem]{Proposition}
\newtheorem{corollary}[theorem]{Corollary}
\newtheorem{lemma}[theorem]{Lemma}
\theoremstyle{definition}

\newtheorem{definition}[theorem]{Definition}
\newtheorem{example}[theorem]{Example}

\newtheorem{remark}[theorem]{Remark}

\newcommand\Nint{\mathbb{N}}
\newcommand\alphabet{\mathcal{A}}

\newcommand\mctree{\mathcal{T}}
\newcommand\tree{t}
\newcommand\pn[1]{p_{#1}}
\newcommand\pna[2]{p_{#1;#2}}
\newcommand\pni[2]{p^{(#1)}_{#2}}
\newcommand\pnai[3]{p^{(#1)}_{#2;#3}}
\newcommand\pnap[2]{q_{#1;#2}}

\newcommand\pnaip[3]{q^{(#1)}_{#2;#3}}
\newcommand\pnab[2]{\bar{p}_{#1;#2}}
\newcommand\pnaib[3]{\bar{p}^{(#1)}_{#2;#3}}
\newcommand\Zint{\mathbb{Z}}
\newcommand\norm[1]{\lvert #1 \rvert}
\newcommand\dhxrightarrow[2][]{%
  \mathrel{\ooalign{$\xrightarrow[#1\mkern4mu]{#2\mkern4mu}$\cr%
  \hidewidth$\rightarrow\mkern4mu$}}
}

\begin{document}

\title{Stem and topological entropy on Cayley trees}
\author[Jung-Chao Ban]{Jung-Chao Ban}
\address[Jung-Chao Ban]{Department of Mathematical Sciences, National Chengchi University, Taipei 11605, Taiwan, ROC.}
\address{Math. Division, National Center for Theoretical Science, National Taiwan University, Taipei 10617, Taiwan. ROC.}
\email{jcban@nccu.edu.tw}

\author[Chih-Hung Chang]{Chih-Hung Chang*}
\thanks{*Author to whom any correspondence should be addressed.} 
\address[Chih-Hung Chang]{Department of Applied Mathematics, National University of Kaohsiung, Kaohsiung 81148, Taiwan, ROC.}
\email{chchang@nuk.edu.tw}

\author[Yu-Liang Wu]{Yu-Liang Wu}
\address[Yu-Liang Wu]{Department of Applied Mathematics, National Chiao Tung University, Hsinchu 30010, Taiwan, ROC.}
\email{s92077.am08g@nctu.edu.tw}

\author[Yu-Ying Wu]{Yu-Ying Wu}
\address[Yu-Ying Wu]{Department of Mathematics, National Central University, Taoyuan 32001, Taiwan, ROC.}
\email{107221003@cc.ncu.edu.tw}

\thanks{Ban and Chang are partially supported by the Ministry of Science and Technology, ROC (Contract No MOST 109-2115-M-004-002-MY2, 109-2115-M-390 -003 -MY3 and 109-2927-I-004-501).}

\begin{abstract}
We consider the existence of the topological entropy of shift spaces on a finitely generated semigroup whose Cayley graph is a tree. The considered semigroups include free groups. On the other hand, the notion of stem entropy is introduced. For shift spaces on a strict free semigroup, the stem entropy coincides with the topological entropy. We reveal a sufficient condition for the existence of the stem entropy of shift spaces on a semigroup. Furthermore, we demonstrate that the topological entropy exists in many cases and is identical to the stem entropy.
\end{abstract}
\maketitle
\tableofcontents

\section{Introduction}
Many simplified mathematical models were proposed to understand phase transition; one of the most famous ones refers to the Ising model. Consider a ferromagnetic metal piece, which consists of a massive number of atoms in the thermal equilibrium. Suppose, ideally, that these atoms are located at the sites of a crystal lattice $\mathbb{Z}^d$. Each atom shows a magnetic moment resulted from the angular moments, which is, for a simplified model, only capable of two orientations. The set of configurations is $X = \alphabet^{\mathbb{Z}^d}$, where $\alphabet = \{1, -1\}$ represents the orientations of spins ``up'' and ``down''. The Ising model is defined by specifying a Hamiltonian (or potential) describing the interaction between spins and then studying the corresponding Gibbs states. Remarkably, there is a unique Gibbs state for $d=1$, whereas for $d \geq 3$, there are infinitely many Gibbs states \cite{Georgii-2011}.

The notion of a Gibbs state (or a Gibbs measure) dates back to R.L.~Dobrushin (1968-1969) and O.E.~Lanford and D.~Ruelle (1969), who proposed it as a mathematical description of an equilibrium state of a physical system which consists of a vast number of interacting components \cite{Dobrushin-TPA1968, Dobrushin-FAA1968, Dobrushin-FAA1968a, Dobrushin-FAA1969, LR-CMP1969}. Gibbs and equilibrium states play a crucial role in the theory of thermodynamic formalism for dynamical systems. A classic example is the investigation of uniformly hyperbolic differential systems such as Anosov and Axiom A diffeomorphisms. The orbits of these systems can be encoded as infinite sequences of finite symbols; the collection of these symbolic sequences forms a superior symbolic dynamical system called a shift of finite type (also known as a topological Markov chain). After constructing the invariant measures, the study of their properties is yielded by the construction of equilibrium states in the sense of statistical mechanics, which turn out to be Gibbs states \cite{Bow-1975, Chernov-2002, Keller-1998, Ruelle-1978}. A remarkable fact is that a Gibbs state for a given type of interaction may not be unique; this, in physical systems, means a phase transition. On the other hand, equilibrium states are defined by a variational principle. More specifically, an equilibrium maximizes the system's entropy under the constraint of fixed mean energy. While a Gibbs state is always an equilibrium state, the reverse fails in general. However, equilibrium states are also Gibbs states provided the given potential function is regular enough \cite{Georgii-2011}.

Investigation of Gibbs states for physical models on a Cayley tree has been received considerable attention recently. One of many motivations is that, in the study of the Ising model on a Cayley tree, a new type of phase transition was revealed \cite{MZ-PRL1974, Mueller-Hartmann-ZPB1975, Mueller-Hartmann-ZPB1977}. Additionally, Zachary showed that, for ferromagnetic or anti-ferromagnetic systems on a Cayley tree, either the set of Gibbs states contains a single point or contains infinitely many points \cite{Zachary-AP1983, Zachary-SPA1985}. A classical construction of Gibbs states on a Cayley tree is the method of Markov random field theory and recurrent equations of this theory; new tools such as group theory, information flows on trees, and node-weighted random walks have been implemented in the modern theory of Gibbs states \cite{Rozikov-2013}.

Aside from the physical significance of systems on a Cayley tree, there are fruitful phenomena observed in these chaotic systems from the mathematical aspect. For instance, the topological conjugacy between two superior symbolic systems (shifts of finite type, to be precise) is decidable; irreducible shifts of finite type are chaotic in Devaney's sense; a stronger type of irreducibility is decidable. See \cite{AB-TCS2012, AB-TCS2013, BC-TAMS2017, BC-N2017, BCH+-JSP2019, CCF+-TCS2013, FF-2012} and the references therein for more details.

As a Gibbs measure maximizes the system's entropy on a crystal lattice $\mathbb{Z}^d$, it is of interest whether this remains true for the system on a Cayley tree. The variational principle for a system $T$ is described, when one considers trivial potential, as $h(T) = h(\mu)$, where $\mu$ is a Gibbs measure, and $h(T)$ and $h(\mu)$ stand for the topological and measure-theoretic entropies, respectively. In the theory of symbolic dynamical systems on a crystal lattice, the topological entropy is defined as the limit of orbits' contribution in the ball of finite volumes. Such a definition is well-defined since $\mathbb{Z}^d$ is an amenable group \cite{CC-2010}. However, the absence of F{\o}lner nets in a Cayley tree makes the definition of topological entropy a controversy; Petersen and Salama extended the study of topological entropy to systems on a rooted Cayley tree via the distribution of orbits on balls \cite{PS-TCS2018, PS-DCDS2020}. Notably, this is how a Gibbs measure for a ferromagnetic Ising model on a Cayley tree is constructed (cf.~\cite{Rozikov-2013}).

It is natural to generalize the definition of topological entropy to systems on a Cayley graph, which corresponds to a group. Under the circumstance of the existence of topological entropy, the problem of whether a Gibbs state is an equilibrium state then follows. In other words, the existence of the topological entropy is essential for the study of the variational principle of shift spaces on a Cayley graph. This motivates the primary concern of the present investigation.
\begin{flushleft}
    \textbf{Problem.} What kinds of groups ensure the existence of the topological entropy of systems on them?
\end{flushleft}

This paper focuses on the existence of topological entropy of systems on a Cayley graph corresponding to a class of finitely generated infinite semigroups. Let $G$ be a finitely generated infinite semigroup and let $S \subset G$ be a finite generating subset of $G$. Suppose a binary matrix $K$ indexed by $S$ describes the relations between any two generators. To be more specific, for $s, s' \in S$, $K(s, s') = 0$ if and only if $s s' = 1_G$, the identity element of $G$. Observe that the Cayley graphs of these groups include the rooted Cayley trees and the Bethe lattices (i.e., regular Cayley trees). In \cite{BCH-JAC2020}, the authors demonstrated that the topological entropy exists for Markov shifts (or topological Markov chains) on a Fibonacci-Cayley tree, which is a Cayley graph corresponding to a semigroup whose growth rate is the golden mean.

The investigation starts with introducing the notion of the \emph{stem entropy} of a shift space on $G$, which, roughly speaking, represents the contribution of complexity on each branch. The existence of the stem entropy follows from the condition that the matrix $K$ is primitive (\textbf{Theorem \ref{same}}) or is irreducible (\textbf{Proposition \ref{irreducible_generalization}}). Beyond that, we demonstrate the coincidence of the stem entropy and the topological entropy provided $K$ is a full matrix, which is equivalent to the condition that $G$ is a free semigroup.

Section 4 applies the existence of the stem entropy to demonstrate whether the topological entropy of a shift space on $G$ exists. Firstly, \textbf{Theorem \ref{thm:entropy_comparison}} reveals that the topological entropy of Markov shifts on $G$ exists provided $K$ has a full row, which is a generalization of \cite{BCH-JAC2020}. Another main result in this section addresses that, suppose the summation of each row of $K$ is identical, the topological entropy of a hom Markov shift on $G$ exists (\textbf{Theorem \ref{thm:hom-shift_free_group}}); hom shifts, initiated from physical systems, are symmetric and isotropic Markov shifts (cf.~\cite{CM-PJM2018}). It is remarkable that the topological entropy, once it exists, coincides with the stem entropy (Theorems \ref{thm:entropy_comparison} and \ref{thm:hom-shift_free_group} and Proposition \ref{prop:top=stem-Markov-free-group}). On the other hand, the regular Cayley tree satisfies the structure required in Theorem \ref{thm:hom-shift_free_group}. An elegant examination of the hard square shift (or the golden mean shift) on a binary free group goes to Piantadosi \cite{Piantadosi-DCDS2008}.
\begin{figure}[h]
    \centering
    \includegraphics[scale=0.6]{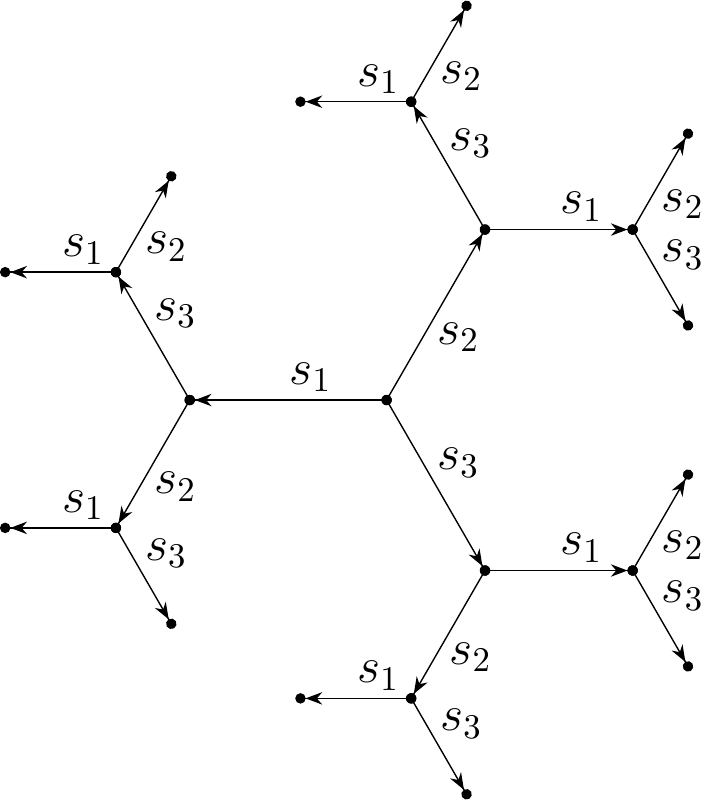}
    \caption{The Bethe lattice of order 3}
    \label{fig:Z_2_free_product_intro}
\end{figure}

Section 5 is devoted to graph representation of Markov shifts on $G$. Under the restriction of Markov shifts on a Cayley graph, the previous sections' results are unified as one. \textbf{Theorem \ref{thm:group_limit}} manifests that, generally speaking, if the graph representation of the system is strongly connected and has a pivot, then the topological entropy exists. Most importantly, the strong connectedness of the graph is equivalent to the irreducibility of a one-dimensional Markov shift; a strongly connected graph representation has a pivot if and only if its associated one-dimensional Markov shift is mixing. It is seen that these conditions are both decidable (\textbf{Proposition \ref{prop:pivot_checkable}}).

Some numerical experiments are carried out in the Appendix for the existence of topological entropy of general systems on $G$. Further elucidation is under preparation.

\section{Preliminaries}

Let $G$ be a finitely generated semigroup and $S_k = \{s_1, s_2, \cdots, s_k\} \subset G$ a finite generating subset of $G$. Suppose the relations between the generators of $G = \langle S_k | R \rangle$ are represented by a binary matrix $K$ indexed by $S_k$ as $R = \{s_i s_j: K(s_i, s_j) = 0\}$. In other words,
$$
s_i s_j = 1_G \quad \text{if and only if} \quad K(s_i, s_j) = 0,
$$
where $1_G$ is the identity element of $G$. The (right) \emph{Cayley graph} of $G$ with respect to $S_k$ is the directed graph $\mctree$ such that the vertex set is $G$ and the edge set is $E = \{(g, gs): g \in G, s \in S_k\}$. It follows that $\mctree$ is an infinite tree. On the other hand, every $g \in G$ has a unique minimal representation $g = g_1 g_2 \cdots g_n$ (with respect to $S_k$). The \emph{length} of the $g$, written as $\norm{g}$, is defined by
$$
|g| = \min \{n: g = g_1 g_2 \cdots g_n, g_i \in S_k\}.
$$
Obviously, $1_G$ is the only element of length $0$, and $g = g_1 g_2 \cdots g_n$, $g_i \in S_k$, is the unique minimal representation if $|g| = n$ and $K(g_i,g_{i+1})=1$ for $1 \leq i \leq n-1$. Such a minimal representation is assumed for every element throughout the paper unless mentioned otherwise. 

\begin{example}
(a) Let $S_2 = \{a, b\}$ and $K = E_2$ the full $2 \times 2$ matrix. Then $G$ is a free semigroup whose Cayley graph $\mctree$ is a binary rooted tree.

(b) Let $S_4 = \{s_1, s_2, s_3, s_4\}$ and $K \in \{0, 1\}^{4 \times 4}$ given by $K(s_i, s_j) = 0$ if and only if $i + j$ is even and $i \neq j$. Then $G = F_2$ is a free group of rank $2$. Figure \ref{fig:regular_tree_1} presents part of the Cayley graph of $F_2$.

(c) Suppose $S_k = \{s_1, s_2, \ldots, s_k\}, k \geq 2$, and $K$ is the $k \times k$ binary matrix given by $K(s_i, s_j) = 0$ if and only if $i = j$. It is seen that $G$ is a free product of $k$ cyclic groups of the second order, and its Cayley graph $\mctree = \Gamma^{k-1}$ is the Bethe lattice (regular Cayley tree) of order $k$.

(d) Let $S_2 = \{a, b\}$ and $K = \begin{pmatrix}
1 & 1 \\
1 & 0
\end{pmatrix}$ the golden mean matrix. The Cayley graph of $G$ is called the Fibonacci-Cayley tree in \cite{BCH-JAC2020} since the growth rate of $\{g \in G: |g| \leq n\}$ is the golden mean.
\end{example}

It is noteworthy that this class of semigroups covers many interesting examples with the exceptions of some simple ones nevertheless, such as $\mathbb{N}^2$ or cyclic groups.

\subsection{Shift spaces on a Cayley tree}

Let $\alphabet$ be a set of finite alphabet. A \emph{labeled tree} (or \emph{configuration}) is a function $\tree: G \to \alphabet$ for which $\tree_g:=\tree(g)$ is the label attached to $g \in G$, and the set $\alphabet^G$ consisting of all labeled trees is called the \emph{full tree shift} or \emph{full shift} on $G$. A \emph{pattern} is a function $u:\alphabet^{H} \to \alphabet$ for some finite set $H \subset G$, where $s(u) := H$ is the \emph{support} of $u$. We say that a pattern $u$ is \emph{accepted} by $\tree \in \alphabet^G$ if there exists $g \in G$ such that $\tree_{g h} = u_h$ for every $h \in s(u)$; otherwise, $t$ \emph{rejects} $u$. A subset $X \subseteq \alphabet^G$ is a \emph{tree shift} (or \emph{shift space on} $G$) if there exists a set of patterns $\mathcal{F}$ such that $t$ rejects any $u \in \mathcal{F}$ for all $t \in X$. We write $X = \mathsf{X}_{\mathcal{F}}$ and call $\mathcal{F}$ a forbidden set for $X$. A tree shift $X$ is a \emph{tree shift of finite type} (TSFT) if $X = \mathsf{X}_{\mathcal{F}}$ for some finite forbidden set $\mathcal{F}$. Let $\mathbf{A}=(A_1,A_2, \ldots, A_k)$ be a $k$-tuple of binary matrices indexed by $\alphabet$. A \emph{Markov tree shift} $X_{\mathbf{A}} \subset \alphabet^{G}$ is defined as
\[
X_{\mathbf{A}}:=\{\tree \in \alphabet^{G}: A_i(\tree_g,\tree_{g s_i}) = 1 \text{ for all } g \in G, \norm{g s_i} = \norm{g}+1\}.
\]
It follows from the definition that a Markov tree shift is a TSFT; conversely, every TSFT is topological conjugate with a Markov tree shift \cite{BC-TAMS2017}.

\subsection{Topological entropy and stem entropy}

Let $G$ be a finitely generated semigroup and $\alphabet$ a finite alphabet. Suppose that $X \subset \alphabet^G$ is a tree shift. We introduce the following notions which are fundamental units in the present elaboration.

For $g \in G$ and $n \geq 0$, denote by $\Delta^{(g)}_{n}$ the $n$-ball centered at $g$ as
$$
\Delta^{(g)}_{n} = \{g h: h \in G, |h| \leq n\},
$$
with $\Delta^{(1_G)}_{n}$ simply denoted by $\Delta_{n}$. On the other hand, we define the $n$-semiball centered at $g$ as
$$
\bar{\Delta}^{(g)}_{n} = \{g h: h \in G, |h| \leq n, \text{ and } |gh|=|g| + |h|\}.
$$
Observe that $\bar{\Delta}^{(g)}_n$ is the initial $n$-subtree rooted at $g$. Furthermore, let
$$
\bar{\Delta}^{(s_i)+}_{n} = \{s_i h: h \in G, |h| \leq n, \text{ and } |s_ih|=1 + |h|\} \cup \{1_G\}
$$
denote the $i$th branch of the Cayley graph with the root $1_G$.

\begin{flushleft}
    \textbf{Notation.} Suppose $g \in G, a \in \alphabet$, and $n$ is a nonnegative integer.
\begin{enumerate}
	\item $B_n^{(g)} := \{u \in \alphabet^{\Delta^{(g)}_n}: u \text{ is accepted by some } \tree \in X\}$;
	\item $B_{n;a}^{(g)} := \{u \in B_n^{(g)}: u_g = a\}$;
	\item $B_n := B_n^{(1_G)}, B_{n;a} := B_{n;a}^{(1_G)}$;
	\item $C_n^{(g)} := \{u \in \alphabet^{\bar{\Delta}^{(g)}_n}: u \text{ is accepted by some } \tree \in X\}$;
	\item $C_n^{(s_i)+} := \{u \in \alphabet^{\bar{\Delta}^{(s_i)+}_n}: u \text{ is accepted by some } \tree \in X\}$;
	\item $C_{n;a}^{(g)} := \{u \in C_n^{(g)}: u_{g} = a\}$;
	\item $C_{n;a}^{(s_i)+} := \{u \in C_n^{(s_i)+}: u_{1_G} = a\}$;
	\item $\pni{g}{n} := |C_n^{(g)}|, \pnai{g}{n}{a} := |C_{n;a}^{(g)}|$, $\pn{n} := |B_n|, \pna{n}{a} := |B_{n;a}|$;
	\item $q_n^{(s_i)} := |C_n^{(s_i)+}|, q_{n;a}^{(s_i)} := |C_{n;a}^{(s_i)+}|$.
\end{enumerate}
\end{flushleft}

Suppose $G = F_2$ is a free group generated by $S_4=\{s_1,s_2,s_1^{-1},s_2^{-1}\}$ for instance. Figures \ref{fig:regular_tree_1} and \ref{fig:regular_tree_2} illustrate $C_{2;a}^{(s_1)}$ and $C_{3;a}^{(s_1)+}$ respectively.

\begin{figure}[t]
	\centering
	\includegraphics[width=0.7\textwidth]{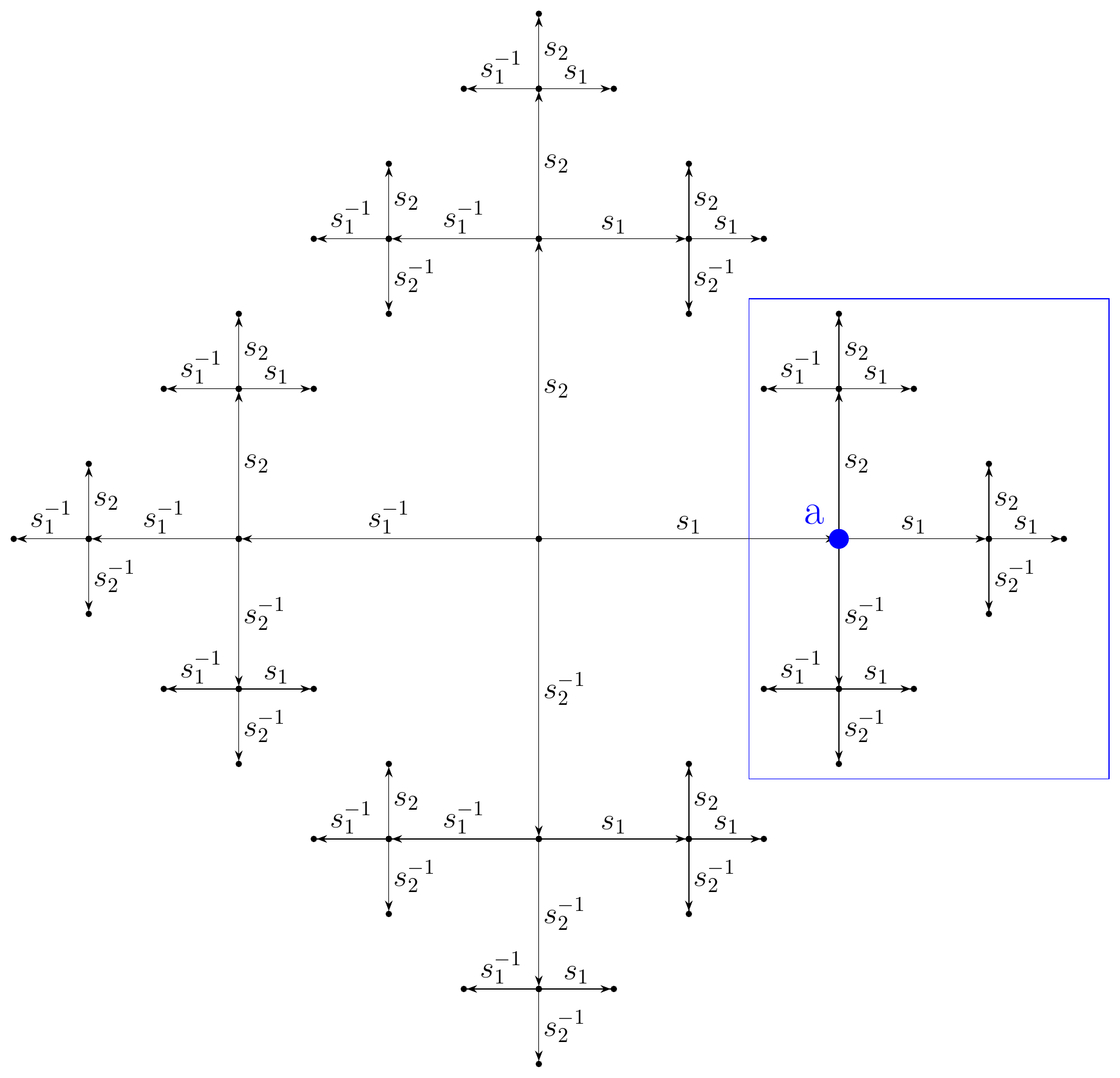}
	\caption{The support of patterns in $C_{2;a}^{(s_1)}$ is a $2$-semiball centered at $s_1$}
	\label{fig:regular_tree_1}
\end{figure}
\begin{figure}[t]
	\centering
	\includegraphics[width=0.7\textwidth]{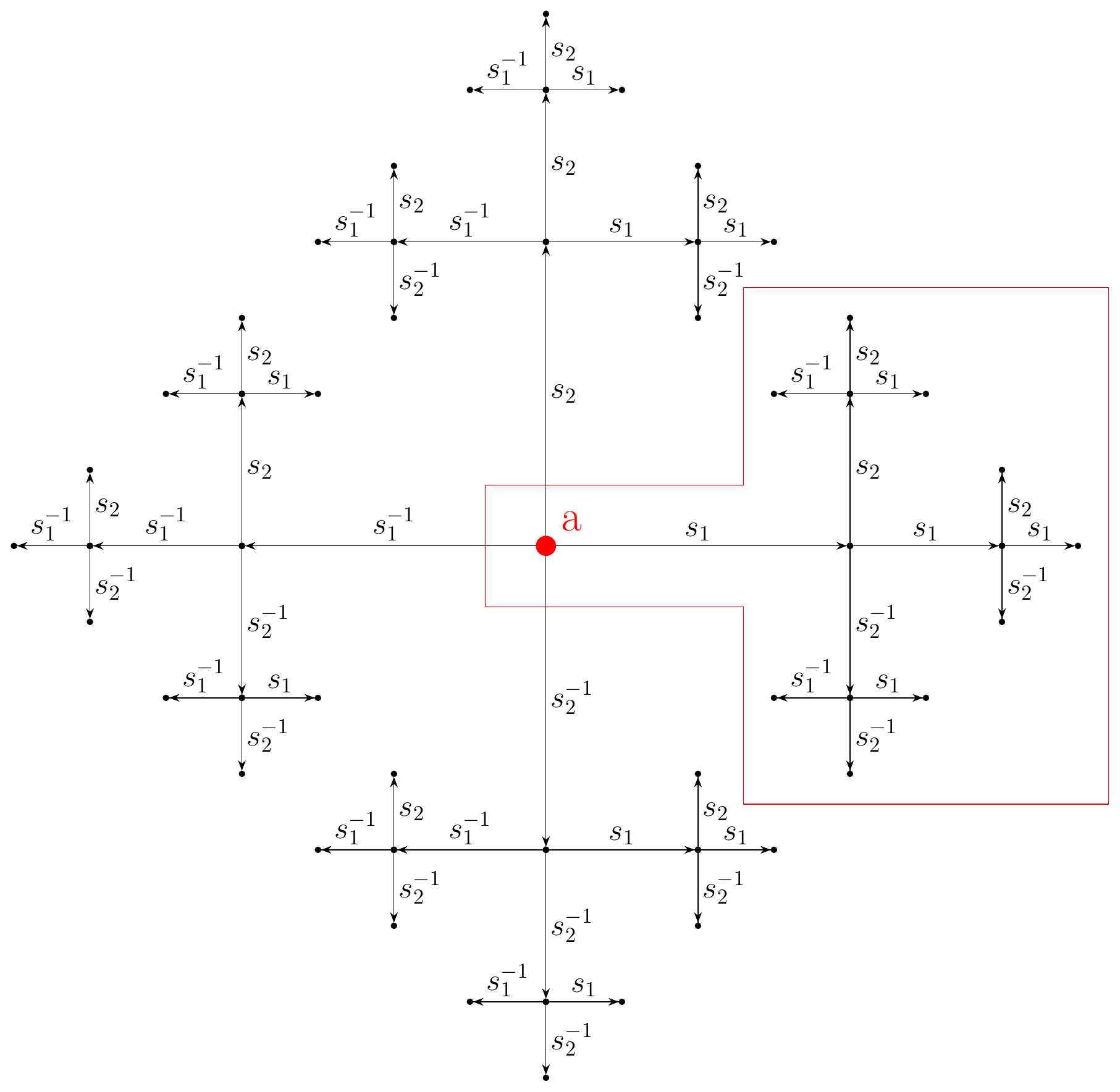}
	\caption{The support of patterns in $C_{3;a}^{(s_1)+}$}
	\label{fig:regular_tree_2}
\end{figure}

\begin{definition}
Suppose $G$ is a finitely generated semigroup and $\alphabet$ is a finite alphabet. Let $X \subset \alphabet^G$ be a tree shift and $g \in G$.
\begin{enumerate}
    \item[(a)] The \emph{$i$th-stem entropy} of $X$ is defined as
\begin{equation}\label{eq:ith-stem-entropy}
h^{(s_i)}=h^{(s_i)}(X):=\limsup_{n \to \infty} \frac{\log \pni{s_i}{n}}{\norm{\bar{\Delta}^{(s_i)}_n}}.
\end{equation}
The \emph{stem entropy}, denoted by $h^{(s)}$, of $X$ exists if $h^{(s_i)} = h^{(s_j)}$ for all $i, j$.
    \item[(b)] The \emph{topological entropy} of $X$ is defined as
\begin{equation}\label{eq:topological-entropy}
h=h(X):=\lim_{n \to \infty} \frac{\log \pn{n}}{\norm{\Delta_n}}
\end{equation}
provided the limit exists.
\end{enumerate}
\end{definition}

\begin{remark}\label{rmk:free-semigroup-stem-equal-topological}
Suppose that $G$ is a strict semigroup; that is, no element in $G$ has an inverse element. A straightforward examination indicates that $h^{(s)} = h$ provided $h^{(s)}$ exists. Indeed, $B_n^{(g)} = C_n^{(g)}$ since $\Delta_n^{(g)} = \bar{\Delta}_n^{(g)}$ for $g \in G$ and $n \geq 0$. Later, Theorem \ref{same} yields a sufficient condition for the existence of the stem entropy for a class of semigroups.
\end{remark}

Suppose that $G = \langle S_k | \rangle$ is a strict free semigroup of rank $k$; that is, $K = E_k$ is a full $k \times k$ matrix. The Cayley graph of $G$ is an infinite rooted tree such that every node has $k$ children. Petersen and Salama \cite{PS-TCS2018, PS-DCDS2020} demonstrated that the topological entropy \eqref{eq:topological-entropy} of a tree shift (i.e., a shift space on an infinite rooted tree) exists and that
\[
h = \inf_{n \to \infty} \frac{\log \pn{n}}{\norm{\Delta_n}}.
\]
In \cite{Piantadosi-DCDS2008}, Piantadosi considered the golden mean shift $X_{\mathbf{A},\mathbf{A}^t}$ on $F_2$, where $\mathbf{A} = (A, A)$ with $A = \begin{pmatrix}
1 & 1 \\ 1 & 0
\end{pmatrix}$
and $\mathbf{A}^t = (A^t, A^t)$ with $A^t$ the transpose of $A$. Recall that $F_2 = \langle S_4 | R \rangle$ such that $R$ is determined by $K(s_i, s_j) = 0$ if and only if $i + j$ is even and $i \neq j$. Piantadosi demonstrated the existence of the topological entropy of $X_{\mathbf{A},\mathbf{A}^t}$ via estimating the growth rate of $q_n^{(s_i)}$. The present paper generalizes Piantadosi's result to a class of Markov tree shifts on $F_l$ for $l \geq 2$. More precisely, we show that the limit of the stem entropy exists (Theorems \ref{same} and \ref{exists_inf}). Additionally, the topological entropy coincides with the stem entropy (Proposition \ref{prop:hom-shift_free_group}).

\section{Existence of Stem Entropy}

This section aims at the exposition of the existence of the stem entropy. A straightforward examination derives that the stem entropy, which does exist, is nothing more than the topological entropy provided $G$ is a strict semigroup. Therefore, the notion of the stem entropy can be seen as an extended discussion of the topological entropy whenever $G$ is not a strict semigroup. On the other hand, it is of interest to the interaction between the stem and topological entropies.

Let $G = \langle S_k | R \rangle$ be a finitely generated semigroup with generating subset $S_k = \{s_1, s_2, \ldots, s_k\}$. Suppose that the relation set $R$ is represented by a $k \times k$ binary matrix $K$ as follows:
$$
s_i s_j \in R \quad \text{if and only if} \quad K(s_i, s_j) = 0.
$$
With abusing the notation, we write $G = \langle S_k | K \rangle$ to specify the equivalence of $R$ and $K$. Let $\mathcal{A}$ be a finite alphabet and $X \subset \mathcal{A}^G$ a tree shift. The main result, existence of the stem entropy, of this section is split into two theorems. The following theorem reveals a class of semigroups such that the stem entropy of a shift space on which exists. Additionally, Theorem \ref{exists_inf} demonstrates that the limit in \eqref{eq:ith-stem-entropy} also exists. For the sake of simplification, the notations $S_k, K, G$, and $\mathcal{A}$ satisfy those conditions above in the remainder of this paper unless otherwise specified.

\begin{theorem}\label{same}
Suppose that $G = \langle S_k | K \rangle$ is a finitely generated semigroup, and $X \subset \mathcal{A}^G$ is a shift space on $G$. If $K$ is primitive, then the stem entropy of $X$ exists. In other words, for $1\leq i, j\leq k$,
\[
	\limsup_{m\rightarrow\infty}\frac{\log{\pni{s_i}{m}}}{|\bar{\Delta}^{(s_i)}_{m}|}=\limsup_{m\rightarrow\infty}\frac{\log{\pni{s_j}{m}}}{|\bar{\Delta}^{(s_j)}_{m}|}. \tag{A1} \label{eq:same}
\]
\end{theorem}

The following series of lemmas are prerequisite for proving the theorem. We start with a property possessed by a primitive matrix. Recall that a nonnegative matrix is primitive if it is eventually positive.

\begin{lemma}\label{eig}
Let $N$ be a $k\times k$ primitive binary matrix and let $\mu$ be its largest eigenvalue. Then, for $1\leq i, j\leq k$, there exists $c=c(i, j)>0$ such that
\[
	\lim_{n\rightarrow\infty}\frac{N^n(i, j)}{c \mu^n}=1.
\]
Furthermore, if $\mu'$ is an eigenvalue of $N$ such that the eigenspace corresponding to $\mu'$ contains a positive vector, then $\mu'>1$.
\end{lemma}

\begin{proof}
The Lemma is a consequence of the Perron-Frobenius theorem. The proof of the asymptotic behavior of $N^n$ is could be found in \cite[Theorem 4.5.12]{LM-1995}. As for $\mu' > 1$, \cite[Theorem 4.2.3]{LM-1995} assures that $\mu'=\mu$, and $N$ is primitive implies $N^{n}(i,j)$ tends to infinity. These together with \cite[Theorem 4.5.12]{LM-1995} leads to $\mu' > 1$.
\end{proof}

\begin{lemma}\label{lim}
Let $\{a_n\}_{n=1}^{\infty}$, $\{c_n\}_{n=1}^{\infty}$ be real sequences and $\{b_n\}_{n=1}^{\infty}$, $\{d_n\}_{n=1}^{\infty}$ be positive real sequences. Suppose
$$
\lim_{n\rightarrow\infty}\frac{a_n}{b_n} = \lim_{n\rightarrow\infty}\frac{c_n}{d_n} = L.
$$
Then
	\[
		\lim_{n\rightarrow\infty}\frac{a_n+c_n}{b_n+d_n}=L.
	\]
Suppose, furthermore, that $\lim_{n\rightarrow\infty} \sum_{j=1}^{n} b_j = +\infty$. Then
	\[
		\lim_{n\rightarrow\infty}\frac{\sum_{j=1}^n a_j}{\sum_{j=1}^n b_j}=L.
	\]
\end{lemma}

\begin{proof}
The equality $\lim_{n\rightarrow\infty}\frac{a_n+c_n}{b_n+d_n}=L$ is immediate and thus the proof is omitted. For the second part, to emphasize the importance of $\lim_{n\rightarrow\infty}\frac{\sum_{j=1}^n a_j}{\sum_{j=1}^n b_j}$ we provide the following detailed discussion.

We prove that for all real numbers $M > L$ and $m < L$, $\limsup_{n\rightarrow\infty}\frac{a_1+\dots+a_n}{b_1+\dots+b_n} \le M$ and $\liminf_{n\rightarrow\infty}\frac{a_1+\dots+a_n}{b_1+\dots+b_n} \ge m$, and thus the lemma is proved. By definition of limit superior, there is a positive integer $N_1$ such that $\frac{a_n}{b_n} < M$ for all $n \geq N_1$.\\
Then for $n>N_1$,
\[
	a_n<M b_n
\]
and
\begin{align*}
	\frac{a_1+\dots+a_n}{b_1+\dots+b_n}
	&=\frac{a_1+\dots+a_{N_1}}{b_1+\dots+b_n}+\frac{a_{N_1+1}+\dots+a_n}{b_1+\dots+b_n}\\
	&<\frac{a_1+\dots+a_{N_1}}{b_1+\dots+b_n}+M\frac{b_{N_1+1}+\dots+b_n}{b_1+\dots+b_n}\\
	&<\frac{a_1+\dots+a_{N_1}}{b_1+\dots+b_n}+M.
\end{align*}
Since $N_1$ is fixed and $\lim_{n\rightarrow\infty} \sum_{j=1}^{n} b_j = +\infty$, we have
\[
	\limsup_{n\rightarrow\infty}\frac{a_1+\dots+a_n}{b_1+\dots+b_n}\leq\lim_{n\rightarrow\infty}(\frac{a_1+\dots+a_{N_1}}{b_1+\dots+b_n}+M)=M.
\]
As for the limit inferior part, there exists a positive integer $N_2$ such that $\frac{a_n}{b_n}>m$ for all $n\geq N_2$. Therefore, \begin{align*}
	\displaystyle\frac{a_1+\dots+a_n}{b_1+\dots+b_n}
	&=\frac{a_1+\dots+a_{N_2}}{b_1+\dots+b_n}+\frac{a_{N_2+1}+\dots+a_n}{b_1+\dots+b_n}\\
	&>\frac{a_1+\dots+a_{N_2}}{b_1+\dots+b_n}+m\frac{b_{N_2+1}+\dots+b_n}{b_1+\dots+b_n}.
\end{align*}
By applying $\lim_{n\rightarrow\infty} \sum_{j=1}^{n} b_j = +\infty$ again, we have that $\lim_{n \to \infty} \frac{b_{N_2+1}+\dots+b_n}{b_1+\dots+b_n} = 1$ and that
\[
	\liminf_{n\rightarrow\infty}\frac{a_1+\dots+a_n}{b_1+\dots+b_n}\geq\lim_{n\rightarrow\infty}\left[\frac{a_1+\dots+a_{N_3}}{b_1+\dots+b_n}+m\frac{b_{N_2+1}+\dots+b_n}{b_1+\dots+b_n}\right]=m.
\]
The proof is thus complete.
\end{proof}

The following lemma gives an explicit expression for the number of nodes in the initial $n$-subtree.

\begin{lemma}\label{dec}
Suppose that $G = \langle S_k | K \rangle$ is finitely generated. For $1\leq i\leq k$, $m\geq0$, $n\geq0$ and $q\geq1$, the following statements are true:
\begin{enumerate}[(i)]
	\item
	\[
		|\bar{\Delta}^{(s_i)}_{n}|=1+\sum_{l=1}^{n}\sum_{j=1}^k K^l(s_i,s_j).
	\]
	\item
	\[
		|\bar{\Delta}^{(s_i)}_{n+q(m+1)}|=
		|\bar{\Delta}^{(s_i)}_{n}|+\sum_{l=1}^k\sum_{j=0}^{q-1} K^{n+j(m+1)+1}(s_i, s_l)|\bar{\Delta}^{(s_l)}_{m}|.
	\]
	\item
	\[
		\pni{s_i}{n+q(m+1)}\leq 
		\pni{s_i}{n}\prod_{l=1}^k
		(\pni{s_l}{m})^{\sum_{j=0}^{q-1} K^{n+j(m+1)+1}(s_i, s_l)}.
	\]

\end{enumerate}
\end{lemma}

\begin{proof}
$\textbf{(i)}$
The length of each element in $\bar{\Delta}^{(s_i)}_{n}$ is at most $n$. There is one element of length $0$, and for $1\leq l\leq n$, there exists $\sum_{j=1}^k K^l(s_i, s_j)$ elements of lenth $l$ in $\bar{\Delta}^{(s_i)}_{n}$. Hence there are $1+\sum_{j=1}^k K(s_{i},s_{j})+\dots+\sum_{j=1}^k K^n(s_{i},s_{j})$
elements in total.

$\textbf{(ii)}$
We prove it by induction on $q$. Since $\bar{\Delta}^{(s_i)}_{n+m+1}$ can be decomposed into disjoint union of $1$ copy of $\bar{\Delta}^{(s_i)}_{n}$, $K^{n+1}(s_{i}, s_{1})$ copies of $\bar{\Delta}^{(s_1)}_{m}$,\dots,$K^{n+1}(s_{i},s_{k})$ copies of $\bar{\Delta}^{(s_k)}_{m}$, thus the result holds when $q=1$.

Suppose the statement is true for some $q-1 \in \mathbb{N}$. Applying the result in the induction step, we obtain
\[
	|\bar{\Delta}^{(s_i)}_{n+(q-1)(m+1)+m+1}|=|\bar{\Delta}^{(s_i)}_{n+(q-1)(m+1)}|+\sum_{l=1}^k K^{n+(q-1)(m+1)+1}(s_{i},s_{l})|\bar{\Delta}^{(s_l)}_{m}|.
\]
The induction hypothesis gives that
\[
	|\bar{\Delta}^{(s_i)}_{n+(q-1)(m+1)}|=
	|\bar{\Delta}^{(s_i)}_{n}|+\sum_{l=1}^k\sum_{j=0}^{q-2} K^{n+j(m+1)+1}(s_{i},s_{l})|\bar{\Delta}^{(s_l)}_{m}|.
\]
Therefore,
\begin{align*}
	|\bar{\Delta}^{(s_i)}_{n+q(m+1)}|&=|\bar{\Delta}^{(s_i)}_{n+(q-1)(m+1)+m+1}|\\
	&=|\bar{\Delta}^{(s_i)}_{n+(q-1)(m+1)}|+\sum_{l=1}^k K^{n+(q-1)(m+1)+1}(s_{i},s_{l})|\bar{\Delta}^{(s_l)}_{m}|\\
	&=|\bar{\Delta}^{(s_i)}_{n}|+\sum_{l=1}^k\sum_{j=0}^{q-2} K^{n+j(m+1)+1}(s_{i},s_{l})|\bar{\Delta}^{(s_l)}_{m}|\\ & \qquad \qquad +\sum_{l=1}^k K^{n+(q-1)(m+1)+1}(s_{i},s_{l})|\bar{\Delta}^{(s_l)}_{m}|\\
	&=|\bar{\Delta}^{(s_i)}_{n}|+\sum_{l=1}^k\sum_{j=0}^{q-1} K^{n+j(m+1)+1}(s_{i},s_{l})|\bar{\Delta}^{(s_l)}_{m}|.
\end{align*}
The proof is complete.

$\textbf{(iii)}$
We prove it by induction on $q$.
Recall that $\bar{\Delta}^{(s_i)}_{n+m+1}$ is a disjoint union of $1$ copy of $\bar{\Delta}^{(s_i)}_{n}$, $K^{n+1}(s_{i},s_{1})$ copies of $\bar{\Delta}^{(s_1)}_{m}$, $K^{n+1}(s_{i},s_{2})$ copies of $\bar{\Delta}^{(s_2)}_{m}$ and so on, where the number of acceptable patterns of $\bar{\Delta}^{(s_i)}_{n}$ is $\pni{s_i}{n}$, and the number of acceptable patterns of $\bar{\Delta}^{(s_j)}_{m}$ is $\pni{s_j}{m}$ for $1\leq j\leq k$. The number $\pni{s_i}{n+m+1}$ could not exceed $\pni{s_i}{n}\prod_{l=1}^k(\pni{s_l}{m})^{(K^{n+1})(s_{i},s_{l})}$,  the result is valid when $q=1$.

Now we assume the result holds for some $q-1 \in \mathbb{N}$. Then
\begin{align*}
	\pni{s_i}{q(m+1)+n}&=\pni{s_i}{(q-1)(m+1)+n+m+1}\\
	&\leq \pni{s_i}{(q-1)(m+1)+n}\prod_{l=1}^k(\pni{s_l}{m})^{(K^{(q-1)(m+1)+n+1})(s_{i},s_{l})}\\
	&\leq \pni{s_i}{n}\prod_{l=1}^k(\pni{s_l}{m})^{\sum_{j=0}^{q-2}(K^{n+j(m+1)+1})(s_{i},s_{l})}\prod_{l=1}^k{(\pni{s_l}{m}})^{(K^{(q-1)(m+1)+n+1})(s_{i},s_{l})}\\
	&\leq \pni{s_i}{n}\prod_{l=1}^k(\pni{s_l}{m})^{\sum_{j=0}^{q-1}(K^{n+j(m+1)+1})(s_{i},s_{l})},
\end{align*}
the proof is complete.
\end{proof}

Aside from the elaboration of Lemmas \ref{eig}-\ref{dec}, the following lemma, which plays a crucial role in the proof of Theorem \ref{same}, further portrays the composition of every $m$-subtree in terms of all $n$-subtree when $m \ge n$.

\begin{lemma}\label{lim2}
Suppose that $G = \langle S_k | K \rangle$ is finitely generated and $K$ is primitive. For $m\geq0$ and $1\leq i, j\leq k$, the following statements are true:
\begin{enumerate}[(i)]
	\item
	\[
		\lim_{n\rightarrow\infty}\frac{|\bar{\Delta}^{(s_j)}_{n}|}{|\bar{\Delta}^{(s_i)}_{n+m+1}|}>0\hbox{ and }\sum_{l=1}^k\lim_{n\rightarrow\infty}\frac{K^{m+1}(s_{i},s_{l})|\bar{\Delta}^{(s_l)}_{n}|}{|\bar{\Delta}^{(s_i)}_{n+m+1}|}=1.
	\]
	\item There exists $\gamma>0$ such that 
	\[
		\lim_{q\rightarrow\infty}\frac{\sum_{l=0}^{q-1} K^{r+l(m+1)+1}(s_{i},s_{j})}{|\bar{\Delta}^{(s_i)}_{q(m+1)+r}|}=\frac{\gamma}{\lambda^{m+1}-1}\hbox{ for all }r\geq0.
	\]
	\item For all $r\geq0$, 
	\[
		\sum_{j=1}^k\lim_{q\rightarrow\infty}\frac{\sum_{l=0}^{q-1} K^{r+l(m+1)+1}(s_{i},s_{j})|\bar{\Delta}^{(s_j)}_{m}|}{|\bar{\Delta}^{(s_i)}_{q(m+1)+r}|}=1.
	\]
\end{enumerate}
\end{lemma}

\begin{proof}
$\textbf{(i)}$
Let $m \ge 0$, $1\leq i, j\leq k$ be given. For $1\leq l\leq k$, from Lemma \ref{eig} we know there are positive numbers $a_l$ and $b_l$ such that 
\begin{equation*}
	\lim_{n\rightarrow\infty}\frac{K^n(s_{j},s_{l})}{a_l\lambda^n}=1
\end{equation*}
and
\begin{equation*}
	\lim_{n\rightarrow\infty}\frac{K^n(s_{i},s_{l})}{b_l\lambda^n}=1.
\end{equation*}
Thus using Lemma \ref{lim}, we have
\begin{equation*}
	\lim_{n\rightarrow\infty}\frac{\sum_{l=1}^{k}K^n(s_{j},s_{l})}{(\sum_{l=1}^{k}a_l)\lambda^n}=1
\end{equation*}
and
\begin{equation*}
	\lim_{n\rightarrow\infty}\frac{\sum_{l=1}^{k} K^{n+m+1}(s_{i},s_{l})}{(\sum_{l=1}^{k}b_l)\lambda^{n+m+1}}=1.
\end{equation*}
Since $K$ is a primitive $\{0, 1\}$-matrix, Lemma \ref{eig} ensures that the largest eigenvalue of $K$ is greater than $1$. Therefore we may apply Lemma \ref{lim} and obtain
\begin{equation}\label{001}
	\lim_{n\rightarrow\infty}\frac{|\bar{\Delta}^{(s_j)}_n|}{\sum_{s=1}^n\sum_{l=1}^ka_l\lambda^s}=
	\lim_{n\rightarrow\infty}\frac{1+\sum_{s=1}^n\sum_{l=1}^k  K^s(s_{j},s_{l})}{\sum_{s=1}^n\sum_{l=1}^k a_l\lambda^s}=1
\end{equation}
and
\begin{equation}\label{002}
	\lim_{n\rightarrow\infty}\frac{|\bar{\Delta}^{(s_i)}_{n+m+1}|}{\sum_{s=1}^{n+m+1}\sum_{l=1}^kb_l\lambda^s}=
	\lim_{n\rightarrow\infty}\frac{1+\sum_{s=1}^{n+m+1}\sum_{l=1}^k K^s(s_{i},s_{l})}{\sum_{s=1}^{n+m+1}\sum_{l=1}^kb_l\lambda^s}=1.
\end{equation}
We also consider
\begin{equation}\label{003}
	\lim_{n\rightarrow\infty}\frac{\sum_{s=1}^n\sum_{l=1}^ka_l\lambda^s}{\sum_{s=1}^{n+m+1}\sum_{l=1}^kb_l\lambda^s}=
	\lim_{n\rightarrow\infty}\frac{\sum_{l=1}^ka_l}{\sum_{l=1}^kb_l}\frac{\frac{\lambda(\lambda^n-1)}{\lambda-1}}{\frac{\lambda(\lambda^{n+m+1}-1)}{\lambda-1}}=
	\frac{\sum_{l=1}^ka_l}{\lambda^{m+1}\sum_{l=1}^kb_l}.
\end{equation}
The existence of the limit of $\{|\bar{\Delta}^{(s_j)}_n|/|\bar{\Delta}^{(s_i)}_{n+m+1}|\}_{n=1}^\infty$ follows from (\ref{001}), (\ref{002}) and (\ref{003}). We also have
\begin{align*}
	\lim_{n\rightarrow\infty}\frac{|\bar{\Delta}^{(s_j)}_n|}{|\bar{\Delta}^{(s_i)}_{n+m+1}|}&=
	\lim_{n\rightarrow\infty}\frac{|\bar{\Delta}^{(s_j)}_n|}{\sum_{s=1}^n\sum_{l=1}^ka_l\lambda^s}
	\lim_{n\rightarrow\infty}\frac{\sum_{s=1}^{n+m+1}\sum_{l=1}^kb_l\lambda^s}{|\bar{\Delta}^{(s_i)}_{n+m+1}|}
	\lim_{n\rightarrow\infty}\frac{\sum_{s=1}^n\sum_{l=1}^ka_l\lambda^s}{\sum_{s=1}^{n+m+1}\sum_{l=1}^kb_l\lambda^s}\\
	&=\frac{\sum_{l=1}^ka_l}{\lambda^{m+1}\sum_{l=1}^kb_l}>0.
\end{align*}
From Lemma \ref{dec} (ii) we see that
\[
	|\bar{\Delta}^{(s_i)}_{n+m+1}|=|\bar{\Delta}^{(s_i)}_{m}|+\sum_{l=1}^k K^{m+1}(s_{i},s_{l})|\bar{\Delta}^{(s_l)}_{n}|.
\] 
Hence it yields
\begin{align*}
	\sum_{l=1}^k\lim_{n\rightarrow\infty}\frac{K^{m+1}(s_{i},s_{l})|\bar{\Delta}^{(s_l)}_{n}|}{|\bar{\Delta}^{(s_i)}_{n+m+1}|}&=
	\lim_{n\rightarrow\infty}\frac{\sum_{l=1}^k K^{m+1}(s_{i},s_{l})|\bar{\Delta}^{(s_l)}_{n}|}{|\bar{\Delta}^{(s_i)}_{n+m+1}|}\\&=
	\lim_{n\rightarrow\infty}\frac{|\bar{\Delta}^{(s_i)}_{n+m+1}|-|\bar{\Delta}^{(s_i)}_{m}|}{|\bar{\Delta}^{(s_i)}_{n+m+1}|}\\&=1.
\end{align*}

$\textbf{(ii)}$ Let $r, m\geq0$ and $1\leq i, j\leq k$ be given. Let $b_1$,\dots,$b_k$ be as in the proof of Lemma \ref{lim2} (i). Then
\[
	\lim_{n\rightarrow\infty}\frac{K^n(s_{i},s_{j})}{b_j\lambda^n}=1\hbox{ for }j=1,\dots,k
\]
and thus
\[
	\lim_{n\rightarrow\infty}\frac{\sum_{l=1}^{k} K^n(s_{i},s_{l})}{\sum_{l=1}^kb_l\lambda^n}=1
\]
by Lemma \ref{lim}. Consequently, using Lemma \ref{lim} we obtain
\[
	\lim_{q\rightarrow\infty}\frac{\sum_{l=0}^{q-1} K^{r+l(m+1)+1}(s_{i},s_{j})}{b_j\sum_{l=0}^{q-1}\lambda^{r+l(m+1)+1}}=1
\]
and
\[
	\lim_{q\rightarrow\infty}\frac{|\bar{\Delta}^{(s_i)}_{q(m+1)+r}|}{\sum_{s=1}^k b_s\sum_{l=1}^{q(m+1)+r}\lambda^l}=
	\lim_{q\rightarrow\infty}\frac{1+\sum_{s=1}^{q(m+1)+r}\sum_{l=1}^{k} K^s(s_{i},s_{l})}{\sum_{s=1}^k b_s\sum_{l=1}^{q(m+1)+r}\lambda^l}=1.
\]
Since
\begin{align*}
	\lim_{q\rightarrow\infty}\frac{\sum_{l=0}^{q-1}(K^{r+l(m+1)+1})(s_{i},s_{j})}{|\bar{\Delta}^{(s_i)}_{q(m+1)+r}|}&=
	\lim_{q\rightarrow\infty}\frac{b_j}{\sum_{s=1}^kb_s}\frac{\sum_{l=0}^{q-1}\lambda^{r+l(m+1)+1}}{\sum_{l=1}^{q(m+1)+r}\lambda^l}
	\\&=
	\lim_{q\rightarrow\infty}\frac{b_j}{\sum_{s=1}^kb_s}\frac{\lambda^{r+1}(\lambda^{q(m+1)}-1)}{\lambda^{m+1}-1}\frac{\lambda-1}{\lambda(\lambda^{q(m+1)+r}-1)}\\&=
	\frac{b_j}{\sum_{s=1}^kb_s}\frac{\lambda-1}{\lambda^{m+1}-1} > 0
\end{align*}
and the number $\frac{b_j(\lambda-1)}{\sum_{s=1}^kb_s}$ does not depend on the choice of $r$, the proof is complete.

$\textbf{(iii)}$ Let $r, m\geq0$ and $1\leq i, j\leq k$ be given. Using Lemma \ref{dec} (ii) we see that
\[
	|\bar{\Delta}^{(s_i)}_{r+q(m+1)}|=
	|\bar{\Delta}^{(s_i)}_{r}|+\sum_{j=1}^k\sum_{l=0}^{q-1} K^{r+l(m+1)+1}(s_{i},s_{j})|\bar{\Delta}^{(s_j)}_{m}|.
\]
Therefore
\begin{align*}
	\sum_{j=1}^k\lim_{q\rightarrow\infty}\frac{\sum_{l=0}^{q-1} K^{r+1+l(m+1)}(s_{i},s_{j})|\bar{\Delta}^{(s_j)}_{m}|}{|\bar{\Delta}^{(s_i)}_{q(m+1)+r}|}&=
	\lim_{q\rightarrow\infty}\frac{\sum_{j=1}^k\sum_{l=0}^{q-1} K^{r+1+l(m+1)}(s_{i},s_{j})|\bar{\Delta}^{(s_j)}_{m}|}{|\bar{\Delta}^{(s_i)}_{q(m+1)+r}|}\\&=
	\lim_{q\rightarrow\infty}\frac{|\bar{\Delta}^{(s_i)}_{r+q(m+1)}|-|\bar{\Delta}^{(s_i)}_{r}|}{|\bar{\Delta}^{(s_i)}_{q(m+1)+r}|} =1.
\end{align*}
This derives the desired result.
\end{proof}

With the delivery of Lemma \ref{lim2}, we are at the position of presenting the proof of Theorem \ref{same}.

\begin{proof}[Proof of Theorem \ref{same}]
Since $K$ is a primitive matrix, we may choose a positive integer $n$ such that $K^n$ is a positive matrix. We also assume that
\[
	\limsup_{m\rightarrow\infty}\frac{\log{\pni{s_{I'}}{m}}}{|\bar{\Delta}^{(s_{I'})}_{m}|}=\max_{1\leq j\leq k}\limsup_{m\rightarrow\infty}\frac{\log{\pni{s_j}{m}}}{|\bar{\Delta}^{(s_j)}_{m}|}.
\]
From Lemma \ref{dec} (iii) we see that
\[
	\pni{s_{I'}}{n+m+1}\leq 
	\pni{s_{I'}}{n}\prod_{l=1}^k
	(\pni{s_l}{m})^{K^{n+1}(s_{I'},s_{l})},
\]
which leads to
\[
	\frac{\log{\pni{s_{I'}}{n+m+1}}}{|\bar{\Delta}^{(s_{I'})}_{m+n+1}|}\leq
	\frac{\log{\pni{s_{I'}}{n}}}{|\bar{\Delta}^{(s_{I'})}_{m+n+1}|}+
	\sum_{l=1}^k\frac{K^{n+1}(s_{I'},s_{l})|\bar{\Delta}^{(s_l)}_{m}|}{|\bar{\Delta}^{(s_{I'})}_{m+n+1}|}\frac{\log{\pni{s_l}{m}}}{|\bar{\Delta}^{(s_l)}_{m}|}.
\]
Recall that Lemma \ref{lim2} gives
\[
	\lim_{m\rightarrow\infty}\frac{|\bar{\Delta}^{(s_j)}_{m}|}{|\bar{\Delta}^{(s_{I'})}_{m+n+1}|}>0\hbox{ for }j=1,\dots,k
\]
and
\[
	\sum_{l=1}^k\lim_{m\rightarrow\infty}\frac{K^{n+1}(s_{I'},s_{l})|\bar{\Delta}^{(s_l)}_{m}|}{|\bar{\Delta}^{(s_{I'})}_{m+n+1}|}=1.
\] 
Therefore
\begin{align*}
	&\sum_{l=1}^k\lim_{m\rightarrow\infty}\frac{K^{n+1}(s_{I'},s_{l})|\bar{\Delta}^{(s_l)}_{m}|}{|\bar{\Delta}^{(s_{I'})}_{m+n+1}|}\limsup_{m\rightarrow\infty}\frac{\log{\pni{s_{I'}}{m}}}{|\bar{\Delta}^{(s_{I'})}_{m}|}\\=&
	\limsup_{m\rightarrow\infty}\frac{\log{\pni{s_{I'}}{m}}}{|\bar{\Delta}^{(s_{I'})}_{m}|}\\=&
	\limsup_{m\rightarrow\infty}\frac{\log{\pni{s_{I'}}{m+n+1}}}{|\bar{\Delta}^{(s_{I'})}_{m+n+1}|}\\\leq&
	\limsup_{m\rightarrow\infty}\frac{\log{\pni{s_{I'}}{n}}}{|\bar{\Delta}^{(s_{I'})}_{m+n+1}|}+
	\sum_{l=1}^k\limsup_{m\rightarrow\infty}\frac{K^{n+1}(s_{I'},s_{l})|\bar{\Delta}^{(s_l)}_{m}|}{|\bar{\Delta}^{(s_{I'})}_{m+n+1}|}\frac{\log{\pni{s_l}{m}}}{|\bar{\Delta}^{(s_l)}_{m}|}\\\leq&
	\sum_{l=1}^k\lim_{m\rightarrow\infty}\frac{K^{n+1}(s_{I'},s_{l})|\bar{\Delta}^{(s_l)}_{m}|}{|\bar{\Delta}^{(s_{I'})}_{m+n+1}|}\limsup_{m\rightarrow\infty}\frac{\log{\pni{s_l}{m}}}{|\bar{\Delta}^{(s_l)}_{m}|}
\end{align*}
and thus
\[
	\limsup_{m\rightarrow\infty}\frac{\log{\pni{s_{I'}}{m}}}{|\bar{\Delta}^{(s_{I'})}_{m}|}=\limsup_{m\rightarrow\infty}\frac{\log{\pni{s_l}{m}}}{|\bar{\Delta}^{(s_l)}_{m}|}\hbox{ for }l=1,\dots, k.
\]
This completes the proof.
\end{proof}

Besides the demonstration of the existence of the stem entropy, the following theorem deduces that the limit in the definition of the stem entropy does exist once $h^{(s_i)} = h^{(s_j)}$ for $1 \leq i, j \leq k$.

\begin{theorem} \label{exists_inf}
Suppose that $G = \langle S_k | K \rangle$ is finitely generated, and $X \subseteq \mathcal{A}^G$ is a tree shift. If $K$ is primitive, then the limit of the $i$th-stem entropy of $X$ \eqref{eq:ith-stem-entropy} exists, and 
\[
	\lim_{n\rightarrow\infty}\frac{\log \pni{s_i}{n}}{|\bar{\Delta}^{(s_i)}_n|}= \inf_{n\geq0}\max_{1\leq j\leq k}\frac{\log \pni{s_j}{n}}{|\bar{\Delta}^{(s_j)}_n|} \quad \text{for} \quad 1 \leq i \leq k. 
	\tag{A2} \label{eq:exists_inf}
\]
\end{theorem}

\begin{proof}
Let $1\leq i\leq k$ and $\epsilon>0$ be given. We choose an integer $m>0$ such that
\[
	\frac{\log \pni{s_i}{m}}{|\bar{\Delta}^{(s_i)}_{m}|}<\liminf_{n\rightarrow\infty}\frac{\log{\pni{s_i}{n}}}{|\bar{\Delta}^{(s_i)}_{n}|}+\epsilon
\]
and
\[
	\frac{\log \pni{s_l}{m}}{|\bar{\Delta}^{(s_l)}_{m}|}<\limsup_{n\rightarrow\infty}\frac{\log{\pni{s_l}{n}}}{|\bar{\Delta}^{(s_l)}_{n}|}+\epsilon, \hbox{ for }l\neq i.
\]
For $r\geq0$, $q\geq1$, Lemma \ref{dec} (iii) gives
\[
	\pni{s_i}{r+q(m+1)}\leq 
	\pni{s_i}{r}\prod_{l=1}^k
	(\pni{s_l}{m})^{\sum_{j=0}^{q-1}(K^{r+j(m+1)+1})(s_{i},s_{l})},
\]
which yields
\begin{equation}\label{004}
	\frac{\log \pni{s_i}{q(m+1)+r}}{|\bar{\Delta}^{(s_i)}_{q(m+1)+r}|}\le\frac{\log \pni{s_i}{r}}{|\bar{\Delta}^{(s_i)}_{q(m+1)+r}|}+
	\sum_{l=1}^{k}\sum_{j=0}^{q-1}\frac{K^{r+j(m+1)+1}(s_{i},s_{l})|\bar{\Delta}^{(s_l)}_{m}|}{|\bar{\Delta}^{(s_i)}_{q(m+1)+r}|}\frac{\log \pni{s_l}{m}}{|\bar{\Delta}^{(s_l)}_{m}|}.
\end{equation}
For $l=1,\dots,k$, let $L^{(l)}$ denote the limit of $\{\sum_{j=0}^{q-1} K^{r+j(m+1)+1}(s_{i},s_{l})|\bar{\Delta}^{(s_l)}_{m}|/|\bar{\Delta}^{(s_i)}_{q(m+1)+r}|\}_{q=1}^{\infty}$. From Lemma \ref{lim2} we know that each $L^{(l)}$ is positive and the value of $L^{(1)}+\dots+L^{(k)}$ is $1$. Taking limit superior at both sides of (\ref{004}) we thus obtain
\begin{align*}
	\limsup_{q\rightarrow\infty}\frac{\log \pni{s_i}{q(m+1)+r}}{|\bar{\Delta}^{(s_i)}_{q(m+1)+r}|}&\leq
	\sum_{l=1}^kL^{(l)}\frac{\log \pni{s_l}{m}}{|\bar{\Delta}^{(s_l)}_{m}|}\\
	&<L^{(i)}\big(\liminf_{n\rightarrow\infty}\frac{\log{\pni{s_i}{n}}}{|\bar{\Delta}^{(s_i)}_{n}|}+\epsilon\big)+\sum_{l\neq i}L^{(l)}\big(\limsup_{n\rightarrow\infty}\frac{\log{\pni{s_l}{n}}}{|\bar{\Delta}^{(s_l)}_{n}|}+\epsilon\big)\\
	&=L^{(i)}\big(\liminf_{n\rightarrow\infty}\frac{\log{\pni{s_i}{n}}}{|\bar{\Delta}^{(s_i)}_{n}|}+\epsilon\big)+\sum_{l\neq i}L^{(l)}\big(\limsup_{n\rightarrow\infty}\frac{\log{\pni{s_i}{n}}}{|\bar{\Delta}^{(s_i)}_{n}|}+\epsilon\big).
\end{align*}
Therefore
\begin{align*}
	\limsup_{n\rightarrow\infty}\frac{\log \pni{s_i}{n}}{|\bar{\Delta}^{(s_i)}_{n}|}&=\max_{0\leq r\leq m}\limsup_{q\rightarrow\infty}\frac{\log \pni{s_i}{q(m+1)+r}}{|\bar{\Delta}^{(s_i)}_{q(m+1)+r}|}\\
	&<L^{(i)}\big(\liminf_{n\rightarrow\infty}\frac{\log{\pni{s_i}{n}}}{|\bar{\Delta}^{(s_i)}_{n}|}+\epsilon\big)+\sum_{l\neq i}L^{(l)}\big(\limsup_{n\rightarrow\infty}\frac{\log{\pni{s_i}{n}}}{|\bar{\Delta}^{(s_i)}_{n}|}+\epsilon\big).
\end{align*}
Since $\epsilon$ is arbitrary, the inequality above leads to
\begin{align*}
	\sum_{l=1}^k L^{(l)}\limsup_{n\rightarrow\infty}\frac{\log \pni{s_i}{n}}{|\bar{\Delta}^{(s_i)}_{n}|}
	&=\limsup_{n\rightarrow\infty}\frac{\log \pni{s_i}{n}}{|\bar{\Delta}^{(s_i)}_{n}|}\\
	&\leq L^{(i)}\liminf_{n\rightarrow\infty}\frac{\log{\pni{s_i}{n}}}{|\bar{\Delta}^{(s_i)}_{n}|}+\sum_{l\neq i}L^{(l)}\limsup_{n\rightarrow\infty}\frac{\log{\pni{s_i}{n}}}{|\bar{\Delta}^{(s_i)}_{n}|}, 
\end{align*}
which also gives
\[
	\limsup_{n\rightarrow\infty}\frac{\log \pni{s_i}{n}}{|\bar{\Delta}^{(s_i)}_{n}|}=\liminf_{n\rightarrow\infty}\frac{\log{\pni{s_i}{n}}}{|\bar{\Delta}^{(s_i)}_{n}|}.
\]

It remains to show that the stem entropies equal the infimum. Note that for $1\leq l\leq k$ the value of $L^{(l)}$ does not depend on the choice of $r$. Observe that (\ref{004}) holds for all $m\geq0$. Taking $r=0$ into (\ref{004}) and letting $q$ tends to infinity, we obtain
\begin{align*}
	\lim_{n\rightarrow\infty}\frac{\log \pni{s_i}{n}}{|\bar{\Delta}^{(s_i)}_{n}|}&=\lim_{q\rightarrow\infty}\frac{\log \pni{s_i}{q(m+1)+r}}{|\bar{\Delta}^{(s_i)}_{q(m+1)+r}|}\\
	&\leq\sum_{l=1}^kL^{(l)}\frac{\log \pni{s_l}{m}}{|\bar{\Delta}^{(s_l)}_{m}|}\\
	&\leq\sum_{l=1}^kL^{(l)}\max_{1\leq j\leq k}\frac{\log \pni{s_j}{m}}{|\bar{\Delta}^{(s_j)}_{m}|}\\
	&=\max_{1\leq j\leq k}\frac{\log \pni{s_j}{m}}{|\bar{\Delta}^{(s_j)}_{m}|}.
\end{align*}
Hence
\[
	\lim_{n\rightarrow\infty}\frac{\log \pni{s_i}{n}}{|\bar{\Delta}^{(s_i)}(n)|}= \inf_{m\geq0}\max_{1\leq j\leq k}\frac{\log \pni{s_j}{m}}{|\bar{\Delta}^{(s_j)}(m)|}.
\]
The proof is complete.
\end{proof}
{\color{cyan}
\begin{proposition} \label{irreducible_generalization}
    The assumption of the matrix $K$ could be loosen so that it is irreducible while \eqref{eq:same} and \eqref{eq:exists_inf} are still valid.
\end{proposition}
\begin{proof}
    Let $K$ be irreducible with period $P$. According to the cyclic structure of $K$ discussed in \cite[Section 4.5]{LM-1995}, with a proper permutation in index, $K$ has the following form:
    \begin{equation} \label{cyclic_decomposition}
        \begin{bmatrix}
            O & K_1 & O & \cdots & O & O \\
            O & O & K_2 & \cdots & O & O \\
            \vdots & \vdots & \vdots & \ddots & \vdots & \vdots \\
            O & O & O & \cdots & O & K_{P-1}\\
            K_P & O & O & \cdots & O & O
        \end{bmatrix}.
    \end{equation}
    Furthermore, by recursively defining $K_{n}=K_{n-P}$ for every $n > P$, the matrix $\mathcal{K}_{r}:=K_{r} K_{r+1} \ldots K_{r+P-1}$ is a primitive matrix and the spectral radius $\rho(\mathcal{K}_{r})=\rho(K)^P$. 
    
    The consequence of the above property yields an estimate of the number $L^{(s_i)}_n$ of points in $n$-th level of $\bar{\Delta}^{(s_i)}_{n}$. Suppose $s_i$ corresponds to the row index $I$ in the matrix $K_r$, which has $k_r$ rows in total. Let $\mathbf{e}_I$ be the $k_r$-dimensional column vector with all entries 0 except for the entry index by $I$ being 1
    \begin{align*}
        L^{(s_i)}_n & =\sum_{j=1}^{k} K^n(s_i,s_j) \\
        & = \mathbf{e}_I^T K_r K_{r+1} \ldots K_{r+n-1} \mathbf{1} \\
        & = \mathbf{e}_I^T \mathcal{K}_r^{\lfloor \frac{n}{P} \rfloor} K_{r + P \lfloor \frac{n}{P} \rfloor} \ldots K_{r+n-1} \mathbf{1}.
    \end{align*}
    Thus,
    \begin{equation} \label{level_num}
        0 < \liminf_{n \to \infty} \frac{L^{(s_i)}_n}{\rho(K)^n} \le \limsup_{n \to \infty} \frac{L^{(s_i)}_n}{\rho(K)^n} < \infty,
    \end{equation}
    and by a similar argument used in Lemma \ref{lim}
    \begin{equation} \label{block_num}
        \begin{aligned}
            \liminf_{n \to \infty} \frac{|\bar{\Delta}^{(s_j)}_{n}|}{|\bar{\Delta}^{(s_i)}_{n+m}|} = \liminf_{n \to \infty} \frac{\sum_{l=0}^{n} L^{(s_j)}_l}{\sum_{l=0}^{n+m} L^{(s_i)}_l} \ge \frac{\liminf_{n \to \infty} \frac{L^{(s_j)}_n}{\rho(K)^n}}{\rho(K)^m \cdot \limsup_{n \to \infty} \frac{L^{(s_i)}_{n+m}}{\rho(K)^{n+m}}} > 0.
        \end{aligned}
    \end{equation}
    
    To prove \eqref{eq:same}, let $n$ be a positive integer such that $K^{n+1}(s_i,s_j)>0$ and 
    \[
	\limsup_{m\rightarrow\infty}\frac{\log{\pni{s_{i}}{m}}}{|\bar{\Delta}^{(s_{i})}_{m}|}=\max_{1\leq l\leq k}\limsup_{m\rightarrow\infty}\frac{\log{\pni{s_j}{m}}}{|\bar{\Delta}^{(s_l)}_{m}|},
\]
\[
	\limsup_{m\rightarrow\infty}\frac{\log{\pni{s_{j}}{m}}}{|\bar{\Delta}^{(s_{j})}_{m}|}=\min_{1\leq l\leq k}\limsup_{m\rightarrow\infty}\frac{\log{\pni{s_j}{m}}}{|\bar{\Delta}^{(s_l)}_{m}|}.
\]
    By taking limit superior in $m$ from 
    \[
        \frac{\log{\pni{s_{i}}{n+m+1}}}{|\bar{\Delta}^{(s_{i})}_{m+n+1}|}\leq
        \frac{\log{\pni{s_{i}}{n}}}{|\bar{\Delta}^{(s_{i})}_{m+n+1}|}+
        \sum_{l=1}^k\frac{K^{n+1}(s_{i},s_{l})|\bar{\Delta}^{(s_l)}_{m}|}{|\bar{\Delta}^{(s_{i})}_{m+n+1}|}\frac{\log{\pni{s_l}{m}}}{|\bar{\Delta}^{(s_l)}_{m}|},
    \]
    we derive
    \begin{align*}
    	&
    	\limsup_{m\rightarrow\infty}\frac{\log{\pni{s_{i}}{m}}}{|\bar{\Delta}^{(s_{i})}_{m}|}\\=&
    	\liminf_{m\rightarrow\infty} \frac{K^{n+1}(s_{i},s_{j})|\bar{\Delta}^{(s_j)}_{m}|}{|\bar{\Delta}^{(s_{i})}_{m+n+1}|} \limsup_{m\rightarrow\infty}\frac{\log{\pni{s_j}{m}}}{|\bar{\Delta}^{(s_j)}_{m}|}\\ & \hspace{3em} +(1-\liminf_{m\rightarrow\infty} \frac{K^{n+1}(s_{i},s_{j})|\bar{\Delta}^{(s_j)}_{m}|}{|\bar{\Delta}^{(s_{i})}_{m+n+1}|}) \limsup_{m\rightarrow\infty}\frac{\log{\pni{s_i}{m}}}{|\bar{\Delta}^{(s_i)}_{m}|}.
    \end{align*}
    Equation \eqref{eq:same} follows as a consequence of $\liminf_{m\rightarrow\infty} \frac{K^{n+1}(s_{i},s_{j})|\bar{\Delta}^{(s_j)}_{m}|}{|\bar{\Delta}^{(s_{i})}_{m+n+1}|}>0$, which follows from \eqref{cyclic_decomposition} and \eqref{block_num}.
    
    We divide the proof of \eqref{eq:exists_inf} into two parts. That is, $\limsup_{n\rightarrow\infty}\frac{\log \pni{s_i}{n}}{|\bar{\Delta}^{(s_i)}_n|} \le \liminf_{n\rightarrow\infty}\frac{\log \pni{s_i}{n}}{|\bar{\Delta}^{(s_i)}_n|}$ and $\lim_{n\rightarrow\infty}\frac{\log \pni{s_i}{n}}{|\bar{\Delta}^{(s_i)}_n|} \le \max_{1\leq j\leq k}\frac{\log \pni{s_j}{n}}{|\bar{\Delta}^{(s_j)}_n|}$ for every $n \ge 0$. For the first part, let $C$ be a positive constant depending only on $K$ defined as
    \[
    C:=\begin{cases}
        1, & \text{if } \rho(K) = 1; \\
        \min\limits_{1 \le i \le k}\left(\lim\limits_{m\rightarrow\infty} \frac{K^{m P}(s_{i},s_{i})}{\rho(K)^{m P}} \inf\limits_{m \ge 0}\frac{|\bar{\Delta}^{(s_i)}_{m}|}{\rho(K)^{m+P}} \liminf\limits_{m\rightarrow\infty} \frac{\rho(K)^{m}}{|\bar{\Delta}^{(s_i)}_{m}|}\right), & \text{if } \rho(K) > 1,
    \end{cases}
    \]
    and let $1\leq i\leq k$ and $\epsilon>0$ be given. We choose an integer $N \ge P$ and $m_0 \ge N$ such that for every $m \ge N$
\[
	\frac{\log \pni{s_l}{m}}{|\bar{\Delta}^{(s_l)}_{m}|}<\limsup_{n\rightarrow\infty}\frac{\log{\pni{s_l}{n}}}{|\bar{\Delta}^{(s_l)}_{n}|}+\epsilon, \hbox{ for }l\neq i
\]
and that
\[
	\frac{\log \pni{s_i}{m_0}}{|\bar{\Delta}^{(s_i)}_{m_0}|}<\liminf_{n\rightarrow\infty}\frac{\log{\pni{s_i}{n}}}{|\bar{\Delta}^{(s_i)}_{n}|}+\epsilon.
\]
For every $m \ge 0$, define $r_m=\max\{n P \ge 0: n P + m_0 + 1 \le m\}$, $n_0=\min\{n \ge N: P | n+m_0+2\}$, $P_0=m_0+n_0+2$ and $S_m=\{r_m-n P_0: n \in \mathbb{N}\}$. Thus, for all sufficiently large $m \ge N$, 
\begin{align*}
	\frac{\log \pni{s_i}{m}}{|\bar{\Delta}^{(s_i)}_{m}|} \le & \frac{|\bar{\Delta}^{(s_i)}_{\min S_m}|}{|\bar{\Delta}^{(s_i)}_{m}|} \frac{\log \pni{s_i}{\min S_m}}{|\bar{\Delta}^{(s_i)}_{\min S_m}|}+
	\sum_{l=1}^{k} \sum_{n \in S_m \cup \{r_m\}}\frac{K^{n}(s_{i},s_{l})|\bar{\Delta}^{(s_l)}_{m_0}|}{|\bar{\Delta}^{(s_i)}_{m}|}\frac{\log \pni{s_l}{m_0}}{|\bar{\Delta}^{(s_l)}_{m_0}|} \\ & \hspace{3em} +
	\sum_{l=1}^{k} \sum_{n \in S_m}\frac{K^{n+m_0+1}(s_{i},s_{l})|\bar{\Delta}^{(s_l)}_{n_0}|}{|\bar{\Delta}^{(s_i)}_{m}|}\frac{\log \pni{s_l}{n_0}}{|\bar{\Delta}^{(s_l)}_{n_0}|} \\ & \hspace{3em} +
	\sum_{l=1}^{k}\frac{K^{r_m+m_0+1}(s_{i},s_{l})|\bar{\Delta}^{(s_l)}_{m-(r_m+m_0+1)}|}{|\bar{\Delta}^{(s_i)}_{m}|}\frac{\log \pni{s_l}{m-(r_m+m_0+1)}}{|\bar{\Delta}^{(s_l)}_{m-(r_m+m_0+1)}|} \\\le &
	\left(\sum_{l=1}^{k} \sum_{n \in S_m \cup \{r_m\}}\frac{K^{n}(s_{i},s_{l})|\bar{\Delta}^{(s_l)}_{m_0}|}{|\bar{\Delta}^{(s_i)}_{m}|}\right) \left(\liminf_{n\rightarrow\infty}\frac{\log{\pni{s_i}{n}}}{|\bar{\Delta}^{(s_i)}_{n}|}+\epsilon\right) \\ & \hspace{3em} +
	\left(1-\sum_{l=1}^{k} \sum_{n \in S_m \cup \{r_m\}}\frac{K^{n}(s_{i},s_{l})|\bar{\Delta}^{(s_l)}_{m_0}|}{|\bar{\Delta}^{(s_i)}_{m}|}\right) \left(\limsup_{n\rightarrow\infty}\frac{\log{\pni{s_i}{n}}}{|\bar{\Delta}^{(s_i)}_{n}|}+\epsilon\right)
\end{align*}
Since $K$ is of the form \eqref{cyclic_decomposition},
by taking limit superior in $m$ from both sides it yields
\begin{align*}
    \limsup_{m\rightarrow\infty}\frac{\log{\pni{s_i}{m}}}{|\bar{\Delta}^{(s_i)}_{m}|} \le & \left(\liminf_{m\rightarrow\infty} \sum_{n \in S_m \cup \{r_m\}}\frac{K^{n}(s_{i},s_{i})|\bar{\Delta}^{(s_i)}_{m_0}|}{|\bar{\Delta}^{(s_i)}_{m}|}\right) \left(\liminf_{m\rightarrow\infty}\frac{\log{\pni{s_i}{m}}}{|\bar{\Delta}^{(s_i)}_{m}|}+\epsilon\right) \\ & \hspace{3em} +
	\left(1-\liminf_{m\rightarrow\infty} \sum_{n \in S_m \cup \{r_m\}}\frac{K^{n}(s_{i},s_{i})|\bar{\Delta}^{(s_i)}_{m_0}|}{|\bar{\Delta}^{(s_i)}_{m}|}\right) \left(\limsup_{m\rightarrow\infty}\frac{\log{\pni{s_i}{m}}}{|\bar{\Delta}^{(s_i)}_{m}|}+\epsilon\right).
\end{align*}
In fact, we can show that the coefficient of the convex combination has the following estimate of lower bound:
\begin{align*}
    \liminf_{m\rightarrow\infty} \sum_{n \in S_m \cup \{r_m\}}\frac{K^{n}(s_{i},s_{i})|\bar{\Delta}^{(s_i)}_{m_0}|}{|\bar{\Delta}^{(s_i)}_{m}|} & \ge C > 0.
\end{align*}
To show this we consider when $\rho(K)=1$, $|\Delta^{(s_l)}_m|=m+1$ and thus
\begin{align*}
    \liminf_{m\rightarrow\infty} \sum_{n \in S_m \cup \{r_m\}}\frac{K^{n}(s_{i},s_{i})|\bar{\Delta}^{(s_i)}_{m_0}|}{|\bar{\Delta}^{(s_i)}_{m}|} & = \frac{m_0+1}{m_0+1+n_0+1} \ge \frac{m_0+1}{m_0+1 + 2( m_0+1)} =\frac{1}{3} = C.
\end{align*}
For the $\rho(K)>1$, 
\begin{align*}
    & \quad \liminf_{m\rightarrow\infty} \sum_{n \in S_m \cup \{r_m\}}\frac{K^{n}(s_{i},s_{i})|\bar{\Delta}^{(s_i)}_{m_0}|}{|\bar{\Delta}^{(s_i)}_{m}|} \\
    & \ge \liminf_{m\rightarrow\infty} \frac{K^{r_m}(s_{i},s_{i})|\bar{\Delta}^{(s_i)}_{m_0}|}{|\bar{\Delta}^{(s_i)}_{m}|} \\
    & \ge \lim_{m\rightarrow\infty} \frac{K^{r_m}(s_{i},s_{i})}{\rho(K)^{r_m}} \frac{|\bar{\Delta}^{(s_i)}_{m_0}|}{\rho(K)^{m_0+P}} \liminf_{m\rightarrow\infty} \frac{\rho(K)^{r_m+m_0+P}}{|\bar{\Delta}^{(s_i)}_{r_m+m_0+P}|} \ge C.
\end{align*}
Hence, we have
\begin{align*}
    \limsup_{m\rightarrow\infty}\frac{\log{\pni{s_i}{m}}}{|\bar{\Delta}^{(s_i)}_{m}|} \le & C \left(\liminf_{m\rightarrow\infty}\frac{\log{\pni{s_i}{m}}}{|\bar{\Delta}^{(s_i)}_{m}|}+\epsilon\right) \\ & \hspace{3em} +
	\left(1-C\right) \left(\limsup_{m\rightarrow\infty}\frac{\log{\pni{s_i}{m}}}{|\bar{\Delta}^{(s_i)}_{m}|}+\epsilon\right).
\end{align*}
Because $\epsilon > 0$ is arbitrary, it follows that $\limsup_{n\rightarrow\infty}\frac{\log{\pni{s_i}{n}}}{|\bar{\Delta}^{(s_i)}_{n}|} = \liminf_{n\rightarrow\infty}\frac{\log{\pni{s_i}{n}}}{|\bar{\Delta}^{(s_i)}_{n}|}$. As for the second part, the proof remains the same as that in Theorem \ref{exists_inf}.
\end{proof}
}

\section{Existence of Topological Entropy}

Recall that the definitions of the topological and stem entropies collapse whenever $G$ is a strict semigroup. Theorems \ref{same} and \ref{exists_inf} yield a class of finitely generated semigroups on which the stem entropy of each tree shift exists, following the derived results, this section is devoted to the existence of the topological entropy and the relationship between the topological entropy and stem entropy. We demonstrate the existence of the topological entropy for a class of tree shifts on $G$, and the topological entropy is identical to the stem entropy. The considered class of semigroups contains but is not limited to the class of finitely generated free groups. For the rest of this article, $G = \langle S_k | K \rangle$ is a finitely generated semigroup with primitive matrix $K$.

Let $\mathbf{A} = (A_1, A_2, \ldots, A_k)$ be a $k$-tuple of binary matrices indexed by $\mathcal{A}$. Recall that a Markov tree shift $X_{\mathbf{A}} \subseteq \mathcal{A}^G$ is defined as
$$
X_{\mathbf{A}} = \{t \in \mathcal{A}^G: A_i (t_g, t_{g s_i}) = 1 \text{ for all } g \in G, |g s_i| = |g| + 1\}.
$$
The following theorem indicates that the topological entropy of a Markov tree shift exists provided $K$ has a full row. Moreover, the topological entropy is identical to the stem entropy.

\begin{theorem} \label{thm:entropy_comparison}
        Suppose $K \in \{0,1\}^{k \times k}$ satisfies $\sum_{j=1}^k K(s_i,s_j) = k$ for some $s_i \in S_k$, and $X$ is a Markov tree shift. Then the topological entropy of $X$ exists and $$h = \lim_{n \to \infty} \frac{\log \pn{n}}{\norm{\Delta_n}} = h^{(s)}.$$
\end{theorem}
\begin{proof}
Note that every $n$-block $u \in B_n$ can be uniquely expressed as a $(k+1)$-tuple $(u_{1_G}, u|_{\bar{\Delta}^{(s_1)}_{n-1}}, u|_{\bar{\Delta}^{(s_2)}_{n-1}}, \cdots, u|_{\bar{\Delta}^{(s_k)}_{n-1}})$, and thus $\pn{n} \le \norm{\alphabet} \cdot \prod_{j=1}^k \pni{s_j}{n-1}$. As a consequence,
\begin{equation} \label{eq:top<=stem}
\limsup_{n \to \infty} \frac{\log \pn{n}}{\norm{\Delta_n}} \le \limsup_{n \to \infty} \frac{\norm{\alphabet}}{\norm{\Delta_n}} + \sum_{j=1}^k \frac{\log \pni{s_j}{n-1}}{\norm{\bar{\Delta}^{(s_j)}_{n-1}}} \frac{\norm{\bar{\Delta}^{(s_j)}_{n-1}}}{\norm{\Delta_n}} = h^{(s)}
\end{equation}
holds by applying Theorem \ref{exists_inf}. On the other hand, $\pni{s_i}{n} \le \pn{n}$ holds naturally, which further implies
\begin{equation} \label{eq:top>=stem}
\liminf_{n \to \infty} \frac{\log \pn{n}}{\norm{\Delta_n}} \ge \liminf_{n \to \infty} \frac{\log \pni{s_i}{n}}{\norm{\Delta_n}} = \liminf_{n \to \infty} \frac{\log \pni{s_i}{n}}{\norm{\bar{\Delta}^{(s_i)}_n}} = h^{(s_i)} = h^{(s)}.
\end{equation}
The proof is finished by combining \eqref{eq:top<=stem} and \eqref{eq:top>=stem} above.
\end{proof}

The theorem above asserts the existence of topological entropy of a Markov tree shift on a Fibonacci-Cayley tree, which was revealed in \cite{BCH-JAC2020}.

\begin{corollary}[See \cite{BCH-JAC2020}] \label{cor:golden_mean_entropy} 
Suppose $G$ is generated by $S_2$ with $K=\begin{pmatrix}1 & 1\\ 1 & 0\end{pmatrix}$, and $X$ is a Markov tree shift. Then the topological entropy of $X$ exists and can be calculated via a system of recurrence equations.
\end{corollary}

A Markov tree shift $X_{\mathbf{A}}$ on $G$ is called a \emph{hom Markov tree shift} if $A_i = A_j$ for all $i, j$. From the physical viewpoint, such a system is isotropic and homogeneous; in other words, two symbols are forbidden to sit next to each other in all coordinate directions once they are forbidden in some direction. The class of hom shift spaces plays an important role in the investigation of physical systems. Suppose that the matrix $K$ has a constant row sum. The theorem below reveals that, not only the topological entropy of a hom Markov tree shift exists, the stem entropy and the topological entropy also coincide.

\begin{theorem} \label{thm:hom-shift_free_group}
    Suppose $m=\sum_{j=1}^k K(s_i,s_j)=\sum_{j=1}^k K(s_{i'},s_j)$ for every $1 \le i, i' \le k$ and $X = X_{\mathbf{A}}$ is a hom Markov tree shift. Then the topological entropy exists and $\lim_{n \to \infty} \frac{\log \pn{n}}{\norm{\Delta_n}} = h^{(s)}$.
\end{theorem}
\begin{proof} 
Since $m=\sum_{j=1}^k K(s_i,s_j)=\sum_{j=1}^k K(s_{i'},s_j)$ for every $1 \le i, i' \le k$ and $A_{1}=A_{2}=\cdots=A_{k}=A$, it follows immediately that $\pnaip{s_i}{n}{a} = \pnaip{s_j}{n}{a}$ for every $s_i, s_j \in S_k$, for which we simply denote $\pnap{n}{a}$ in the rest of the proof.
Note that since $x^{\frac{k}{m}}$ is convex, the following inequality holds for every $s_i \in S_k$:
\begin{align*}
        (\pni{s_i}{n})^{\frac{k}{m}}&=(\sum_{a=1}^{\norm{\alphabet}} (\pnap{n}{a})^m)^{\frac{k}{m}} \\
        &=(\norm{\alphabet} \sum_{a=1}^{\norm{\alphabet}} \frac{1}{\norm{\alphabet}} \cdot (\pnap{n}{a})^{m})^{\frac{k}{m}} \\
        &\le {\norm{\alphabet}}^{\frac{k-m}{m}} \sum_{a=1}^{\norm{\alphabet}} (\pnap{n}{a})^{k} \\
        &={\norm{\alphabet}}^{\frac{k-m}{m}} \pn{n}.
\end{align*}
On the other hand, it can be deduced by applying Minkowski inequality that
\begin{equation*}
    (\pni{s_i}{n})^{\frac{k}{m}}=(\sum_{a=1}^{\norm{\alphabet}} (\pnap{n}{a})^m)^{\frac{k}{m}} \ge \sum_{j=1}^{\norm{\alphabet}} (\pnap{n}{a})^{k} = \pn{n}.
\end{equation*}
By combining the inequalities above , it yields that $(\pni{s_i}{n})^{\frac{k}{m}} \ge \pn{n} \ge \norm{\alphabet}^{\frac{m-k}{m}} (\pni{s_i}{n})^{\frac{k}{m}}$ and thus
\begin{align*}
\frac{\log \pni{s_i}{n}}{\norm{\bar{\Delta}^{(s_i)}_n}} \frac{\frac{k}{m}\norm{\bar{\Delta}^{(s_i)}_n}}{\norm{\Delta_n}} + \frac{m-k}{m}\frac{\log \norm{\alphabet}}{\norm{\Delta_n}}
&=\frac{\log \pni{s_i}{n}}{\norm{\bar{\Delta}^{(s_i)}_n}}\frac{\frac{k}{m}\norm{\bar{\Delta}^{(s_i)}_n}}{\frac{k}{m}(\norm{\bar{\Delta}^{(s_i)}_n}-1)+1}+\frac{m-k}{m}\frac{\log \norm{\alphabet}}{\norm{\Delta_n}} \\
&\le \frac{\log \pn{n}}{\norm{\Delta_n}} \\
&\le \frac{\log \pni{s_i}{n}}{\norm{\bar{\Delta}^{(s_i)}_n}}\frac{\frac{k}{m}\norm{\bar{\Delta}^{(s_i)}_n}}{\norm{\Delta_n}}=\frac{\log \pni{s_i}{n}}{\norm{\bar{\Delta}^{(s_i)}_n}}\frac{\frac{k}{m}\norm{\bar{\Delta}^{(s_i)}_n}}{\frac{k}{m}(\norm{\bar{\Delta}^{(s_i)}_n}-1)+1}.
\end{align*}
Since $\lim_{n\rightarrow\infty}\frac{\log \pni{s_i}{n}}{\norm{\bar{\Delta}^{(s_i)}_n}}$ is proved to be $h^{(s)}$ for all $s_i \in S_k$ in Theorem \ref{exists_inf}, the proof is finished.
\end{proof}

\begin{example}
    A class of groups satisfying the assumption of Theorem \ref{thm:hom-shift_free_group} is the Bethe lattice, for which the matrices $K$'s have each diagonal entry 0 and each non-diagonal entry 1. For instance, the Bethe lattice of order 3 is provided in Figure \ref{fig:Z_2_free_product_intro}.
\end{example}

An immediate application of Theorem \ref{thm:hom-shift_free_group} is that the topological entropy of a hom Markov tree shift on a free group exists. Suppose that $\mathbf{A} = (A_1, A_2, \ldots, A_k)$. We denote by $\mathbf{A}^t = (A_1^t, A_2^t, \ldots, A_k^t)$ the $k$-tuple of transpose matrices of $\mathbf{A}$. Theorem \ref{thm:hom-shift_free_group} is further generalized to the following proposition.

\begin{proposition} \label{prop:hom-shift_free_group}
     Let $G = F_k$ be a free group of rank $k$. That is, $G = \langle S_{2k} | K \rangle$ with $K(s_i, s_j) = 0$ if and only if $|i-j| = k$. Suppose $X = X_{\mathbf{A},\mathbf{A}^t}$ is a Markov shift space over $F_k$ with $A_1=A_2=\cdots=A_k=A$ indexed by a finite alphabet $\mathcal{A}$. Then the limit $\lim_{n \to \infty} \frac{\log \pn{n}}{\norm{\Delta_n}}$ exists and equals $h^{(s)}$.
\end{proposition}
\begin{proof}
    For simplicity, we write $\pnap{n}{a}^+ = \pnaip{s_i}{n}{a} = \pnaip{s_j}{n}{a}$, $\pnap{n}{a}^- = \pnaip{s_i^{-1}}{n}{a}=\pnaip{s_j^{-1}}{n}{a}$, $\norm{\bar{\Delta}^{(s_i)}_n}=\norm{\bar{\Delta}^{(s_i^{-1})}_n}=\norm{\bar{\Delta}_n}$ and $\norm{\bar{\Delta}^{(s_i)+}_n}=\norm{\bar{\Delta}^{(s_i^{-1})+}_n}=\norm{\bar{\Delta}^{+}_n}$ in the rest of the proof.
    
    First, we claim that $\limsup_{n \to \infty} \frac{\log \pnap{n}{a}^+}{\norm{\bar{\Delta}^+_n}} \le h^{(s)}$ and that $\limsup_{n \to \infty} \frac{\log \pnap{n}{a}^-}{\norm{\bar{\Delta}^+_n}} \le h^{(s)}$. To show this, note that $\pnap{n}{a}^+ \le \pni{s_1}{n-1}$ and thus $\frac{\log \pnap{n}{a}^+}{\norm{\bar{\Delta}^{+}_n}} \le \frac{\log \pni{s_1}{n-1}}{\norm{\bar{\Delta}^{(s_1)}_{n-1}}}$. The inequality then holds by taking limit superior of both sides, and the same arguments apply to $\pnap{n}{a}^-$.
    
    Now we claim that $\lim_{n \to \infty} \frac{\pn{n}}{\norm{\Delta_n}}$ exists and equals $h^{(s)}$. Since it follows from \eqref{eq:top<=stem} that $\limsup_{n \to \infty} \frac{\log \pn{n}}{\norm{\Delta_n}} \le h^{(s)}$, it is left to show that $\liminf_{n \to \infty} \frac{\log \pn{n}}{\norm{\Delta_n}} \ge h^{(s)}$. Since $\pni{s_1}{n} = \sum_{a \in \alphabet} (\pnap{n}{a}^+)^{k} \cdot (\pnap{n}{a}^-)^{k-1}$, there exists $a_n \in \alphabet$ for each $n$ such that $(\pnap{n}{a_n}^+)^{k} \cdot (\pnap{n}{a_n}^-)^{k-1} \ge \frac{\pni{s_1}{n}}{\norm{\alphabet}}$. Hence, by applying Theorem \ref{exists_inf} and the claim above, for every $\epsilon > 0$ there exists $N \in \Nint$ such that 
    \[
        \pnap{n}{a_n}^+,\pnap{n}{a_n}^- < e^{(h^{(s)} + \epsilon) \norm{\bar{\Delta}^{+}_n}},
    \]
    and that
    \[
        (\pnap{n}{a_n}^+)^{k} \cdot (\pnap{n}{a_n}^-)^{k-1} \ge \frac{1}{\norm{\alphabet}} \pni{s_1}{n} > e^{(h^{(s)}-\epsilon) \norm{\bar{\Delta}_n}},
    \]
    whenever $n \ge N$. This implies
    \begin{align*}
        \pnap{n}{a_n}^- & = \frac{(\pnap{n}{a_n}^+)^{k} \cdot (\pnap{n}{a_n}^-)^{k-1}}{(\pnap{n}{a_n}^+)^{k} \cdot (\pnap{n}{a_n}^-)^{k-2}} \\
        & \ge e^{(h^{(s)} - \epsilon) \norm{\bar{\Delta}_n} - (h^{(s)}+\epsilon) \norm{\bar{\Delta}^{+}_n} (2k-2)} \\
        & = e^{(h^{(s)} - \epsilon) \left[(2k-1) \norm{\bar{\Delta}^{+}_n} - (2k-2)\right] - (h^{(s)}+\epsilon) \norm{\bar{\Delta}^{+}_n} (2k-2)} \\
        & = e^{-(2k-2) (h^{(s)} - \epsilon)} e^{(h^{(s)} - (4k-3)\epsilon ) \norm{\bar{\Delta}^{+}_n}}.
    \end{align*}
    Hence,
    \begin{align*}
        \pna{n}{a_n} & = (\pnap{n}{a_n}^+)^{k} \cdot (\pnap{n}{a_n}^-)^{k-1} \cdot \pnap{n}{a_n}^- \\
        & \ge e^{(h^{(s)} - \epsilon) \norm{\bar{\Delta}_n}} \cdot e^{-(2k-2) (h^{(s)} - \epsilon)} e^{(h^{(s)} - (4k-3)\epsilon ) \norm{\bar{\Delta}^{+}_n}} \\
        & \ge e^{(h^{(s)} - (4k-3) \epsilon) \norm{\bar{\Delta}_n}} \cdot e^{-(2k-2) (h^{(s)} - \epsilon)} e^{(h^{(s)} - (4k-3)\epsilon ) \norm{\bar{\Delta}^+_n}} \\
        & = e^{(h^{(s)} - (4k-3) \epsilon) (\norm{\bar{\Delta}_n}+\norm{\bar{\Delta}^{+}_n})} \cdot e^{-(2k-2) (h^{(s)} - \epsilon)} \\
    \end{align*}
    Hence, one obtains 
    \[
    \liminf_{n \to \infty} \frac{\pn{n}}{\norm{\Delta_n}} \ge \liminf_{n \to \infty} \frac{\pna{n}{a_n}}{\norm{\Delta_n}} \ge h^{(s)}.
    \]
    This finishes the proof.
\end{proof}

Using the same technique as above, one can also obtain a variation of Proposition \ref{prop:hom-shift_free_group} as follows.

\begin{proposition} \label{prop:top=stem-Markov-free-group}
     Suppose $\mathcal{A}$ is a finite alphabet with $|\mathcal{A}| \leq 2 k-1$. Let $X_{\mathbf{A},\mathbf{A}^t}$ be a Markov shift over $F_k$ with $\mathbf{A} = (A_1, A_2, \ldots, A_k)$. Then the topological entropy of $X$ exists and equals $h^{(s)}$.
\end{proposition}
\begin{proof}
    For simplicity, we write $\norm{\bar{\Delta}^{(s_i)}}=\norm{\bar{\Delta}^{(s_i^{-1})}}=\norm{\bar{\Delta}_n}$ and $\norm{\bar{\Delta}^{(s_i)+}}=\norm{\bar{\Delta}^{(s_i^{-1})+}}=\norm{\bar{\Delta}^{+}_n}$ in the rest of the proof.

    By applying the argument in Proposition \ref{prop:hom-shift_free_group}, one obtains that 
    \begin{equation} \label{eq:limsup_upper_bound}
        \limsup \frac{\log \pnaip{s_i}{n}{a}}{\norm{\bar{\Delta}^{(s_i)}_n}} \le h^{(s)}
    \end{equation}
    for every $s_i \in S_{2 k}$. Now we claim that $\lim_{n \to \infty} \frac{\pn{n}}{\norm{\Delta_n}}$ exists and equals $h^{(s)}$. Since it follows from \eqref{eq:top<=stem} that $\limsup_{n \to \infty} \frac{\log \pn{n}}{\norm{\Delta_n}} \le h^{(s)}$, it is left to show that $\liminf_{n \to \infty} \frac{\log \pn{n}}{\norm{\Delta_n}} \ge h^{(s)}$. Since $\pni{z}{n} = \sum_{a \in \alphabet} \prod_{w \ne z^{-1}} \pnaip{w}{n}{a}$ for every $z \in S_{2 k}$, there exists $a_{n;z} \in \alphabet$ for each $n$ such that $\prod_{w \ne z^{-1}} \pnaip{w}{n}{a_{n;z}} \ge \frac{\pni{z}{n}}{\norm{\alphabet}}$. Hence, by applying Theorem \ref{exists_inf} and \eqref{eq:limsup_upper_bound}, for every $\epsilon > 0$ there exists $N \in \Nint$ such that 
    \[
        \pnaip{w}{n}{a_{n;z}} < e^{(h^{(s)} + \epsilon) \norm{\bar{\Delta}^{+}_n}},
    \]
    and that
    \[
        \prod_{w \ne z^{-1}} \pnaip{w}{n}{a_{n;z}} \ge \frac{1}{\norm{\alphabet}} \pni{z}{n} > e^{(h^{(s)}-\epsilon) \norm{\bar{\Delta}_n}},
    \]
    for all $z, w \in S_{2 k}$ and all $n \ge N$. At this moment, it is noteworthy  that the restriction imposed on the dimension of $A_i$ leads to the coincidence of some $a_{n;z_1}=a_{n;z_2}$ ($z_1 \ne z_2$) by the pigeonhole principle, and thus $K(z_2,z_1^{-1})=1$. These two properties together imply that if $u, v$ are admissible patterns in $X_{\mathbf{A},\mathbf{A}^t}$ with $u_{1_G} = a_{n;z_1} = a_{n;z_2} = v_{1_G}$, $s(u)=\bar{\Delta}^{(z_1)}_n$, and $s(v)=\bar{\Delta}^{(z_2)}_n$, then $\overline{u}$ with support $s(\overline{u}) = \Delta_n$, defined as follows, is also a admissible pattern:
    \[
    \overline{u}_g:=\begin{cases}
    v_g, & \text{if } g=z_2^{-1} g', \norm{g} = \norm{z_2^{-1}} + \norm{g'}; \\
    u_g, & \text{otherwise.}
    \end{cases}
    \]
    As a consequence, 
    \[\pna{n}{a_{n;z_1}} = \pnaip{z_1^{-1}}{n}{a_{n;z_2}} \cdot \prod_{w \ne z_1^{-1}} \pnaip{w}{n}{a_{n;z_1}},
    \]
    and
    \begin{align*}
        \pnaip{z_1^{-1}}{n}{a_{n;z_1}} & = \frac{\prod_{w \ne z_2^{-1}} \pnaip{w}{n}{a_{n;z_1}}}{\prod_{w \ne z_1^{-1},w \ne z_2^{-1}} \pnaip{w}{n}{a_{n;z_1}}} \\
        & \ge e^{(h^{(s)} - \epsilon) \norm{\bar{\Delta}_n} - (h^{(s)}+\epsilon) \norm{\bar{\Delta}^{+}_n} (2k-2)} \\
        & = e^{(h^{(s)} - \epsilon) \left[(2k-1) \norm{\bar{\Delta}_n} - (2k-2)\right] - (h^{(s)}+\epsilon) \norm{\bar{\Delta}^{+}_n} (2k-2)} \\
        & = e^{-(2k-2) (h^{(s)} - \epsilon)} e^{(h^{(s)} - (4k-3)\epsilon ) \norm{\bar{\Delta}^{+}_n}}.
    \end{align*}
    Combining all the results above, it follows that
    \begin{align*}
        \pna{n}{a_n} & = \pnaip{z_1^{-1}}{n}{a_{n;z_1}} \cdot \prod_{w \ne z_1^{-1}} \pnaip{w}{n}{a_{n;z_1}} \\
        & \ge e^{(h^{(s)} - \epsilon) \norm{\bar{\Delta}_n}} \cdot e^{-(2k-2) (h^{(s)} - \epsilon)} e^{(h^{(s)} - (4k-3)\epsilon ) \norm{\bar{\Delta}^{+}_n}} \\
        & \ge e^{(h^{(s)} - (4k-3) \epsilon) \norm{\bar{\Delta}_n}} \cdot e^{-(2k-2) (h^{(s)} - \epsilon)} e^{(h^{(s)} - (4k-3)\epsilon ) \norm{\bar{\Delta}^{+}_n}} \\
        & = e^{(h^{(s)} - (4k-3) \epsilon) (\norm{\bar{\Delta}_n}+\norm{\bar{\Delta}^{+}_n})} \cdot e^{-(2k-2) (h^{(s)} - \epsilon)} \\
    \end{align*}
    Hence, one obtains 
    \[
    \liminf_{n \to \infty} \frac{\pn{n}}{\norm{\Delta_n}} \ge \liminf_{n \to \infty} \frac{\pna{n}{a_{n;z_1}}}{\norm{\Delta_n}} \ge h^{(s)}.
    \]
    This finishes the proof.
\end{proof}

\section{Generalization of Mixing Property}

Aside from the straightforward estimation of topological entropy in the previous section, this section studies from an topological perspective the coincidence between stem entropy and topological entropy. In fact, the exposition in the following is inspired by \cite[Proposition 3.1]{PS-TCS2018} and generalizes the idea of mixing property on hom Markov tree shifts on a strict semigroup to that on finitely generated semigroup expressed as $G = \langle S_k | K \rangle$. We begin with defining the following terms.

\begin{definition}
    Let $G = \langle S_k | K \rangle$ be a finitely generated semigroup. Suppose $X = X_{\mathbf{A}} \subseteq \mathcal{A}^G$ is a Markov tree shift on $G$. A \emph{graph representation of $X$} is a directed graph $\mathbf{G} = (\mathbf{V}, \mathbf{E})$ with vertex set $\mathbf{V} = \alphabet \times S_k$ and with edge set $\mathbf{E} = \{((a,s_i),(b,s_j)) \in \mathbf{V} \times \mathbf{V}: K(s_i,s_j)=1, A_j(a,b) = 1\}$.
    \begin{enumerate}[(i)]
    	\item $\mathbf{G}$ is called \emph{strongly connected} if for every $(a,s_i), (b,s_j) \in \mathbf{V}$ there is a walk of length $N$ from $(a,s_i)$ to $(b,s_j)$ in $\mathbf{G}$ (denoted by $(a,s_i) \dhxrightarrow{N} (b,s_j)$) for some $N$ depending on $(a,s_i)$ and $(b,s_j)$.
    	\item A vertex $(a,s_i) \in \mathbf{V}$ is called a \emph{pivot} if there exist $s_j \in S_k$ and an integer $N \in \Nint$ such that every $(b,s_j) \in \mathbf{V}$ admits a walk $(a,s_i) \dhxrightarrow{N} (b,s_j)$.
    \end{enumerate}
\end{definition}

\begin{example}
    Suppose $G = \langle S_2 | K \rangle$ is associated with the matrix $K=\begin{bmatrix}
		1 & 1 \\
		1 & 0
	\end{bmatrix}$ and
    \[
    A_1=\begin{bmatrix}
    	1 & 1 \\
    	1 & 0
    \end{bmatrix}, A_2=\begin{bmatrix}
    	0 & 1 \\
    	1 & 1
    \end{bmatrix}
    \]
    are the adjacency matrices for the shift space $X_{A_1,A_2} \subset \alphabet^G$. Then, the graph representation of $X_{A_1,A_2}$ is defined as in Figure \ref{fig:graph_representation_1}.
    \begin{figure}
        \centering
        \includegraphics[scale=1.5]{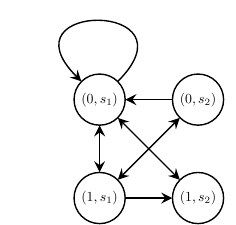}
        \caption{Graph representation of $X_{A_1,A_2}$}
        \label{fig:graph_representation_1}
    \end{figure}
\end{example}

To see the definitions above are related to the mixing property, we prove the following proposition.

\begin{proposition} \label{prop:graph_and_hom_shift}
Suppose that $X_{\mathbf{A}} \subseteq \mathcal{A}^G$ is a hom Markov tree shift, and $\mathbf{G}=(\mathbf{V},\mathbf{E})$ is a graph representation of $X_{\mathbf{A}}$. Then, 
\begin{enumerate}[(i)]
	\item $\mathbf{G}$ is strongly connected if and only if $A$ is irreducible.
	\item $\mathbf{G}$ is strongly connected and contains a pivot if and only if $A$ is primitive.
\end{enumerate}
\end{proposition}
\begin{proof}
\textbf{(i)} It is not hard to see that $A$ is irreducible if $\mathbf{G}$ is strongly connected, since for $(a,s_i), (b,s_i) \in \mathbf{V}$, there exists a walk $(a,s_i) (a_1,s_{i_1}) (a_2,s_{i_2}) \cdots (a_{n-1},s_{i_{n-1}}) (b,s_i)$ and thus $a a_1 a_2 \cdots a_{n} b$ is a word admissible by $A$. We now show the converse, i.e., for $(a,s_i), (b,s_j) \in \mathbf{V}$, there exists a walk $(a,s_i) \dhxrightarrow{M} (b,s_j)$. Since $K$ is a primitive matrix, there exists $N$ such that for every $n \in \Nint$ and $s_i,s_j \in S_k$, there is an admissible word $s_i s_{i_1} s_{i_2} \cdots s_{i_{n-1}} s_j$ by $K$. On the other hand, since $A$ is irreducible, for every $a,b \in \alphabet$ there exists an integer $M \ge N$ and an $M$-word $a a_1 a_2 \cdots a_{N-1} b$ admissible by $A$. This results in a walk $(a_,s_i) (a_1, s_{i_1}) \cdots (a_{M-1},s_{i_{M-1}}) (b,s_j)$ in $\mathbf{G}$. This completes the proof.

\textbf{(ii)} 
First of all, we show that $A$ is primitive if the adjacency matrix $A_\mathbf{G}$ of $G$ is primitive. Indeed, since $A_\mathbf{G}$ is primitive, there exists $N$ such that for all $(a,s_i), (b,s_j)$ and $n \ge N$, there exists a admissible walk $(a,s_i) \dhxrightarrow{n} (b,s_j)$ in $\mathbf{G}$. This naturally yields a $(n+1)$-word admissible by $A$ which starts at $a$ and terminates at $b$.

Secondly, we show that $\mathbf{G}$ is strongly connected and contains a pivot provided $A$ is primitive. To this end, we show every $(a,s_i) \in \mathbf{V}$ is a pivot of $\mathbf{G}$. Since $K$ is a primitive matrix, there exists an integer $N_1$ such that for every $s_j \in S_k$ and every $n \ge N_1$, there exists an $(n+1)$-word admissible by $K$ which starts from $s_i$ and terminates at $s_j$. On the other hand, since $A$ is primitive, there exists $N_2 \ge N_1$ such that for every $b \in \alphabet$ there is a admissible word $a a_1 \cdots a_{N_2-1} b$ by $A$. This implies for all $n \ge N_2$ there is a walk $(a,s_i) (a_1,s_{i_1}) \cdots (a_{n-1}, s_{i_{n-1}}) (a_n,s_{i_n})$ in $\mathbf{G}$. This finishes the proof of our claim. Note since every $(a,s_i)$ is a pivot, irreducibility follows
immediately.

Finally, it remains to show that if $\mathbf{G}$ is strongly connected and contains a pivot, then $A_\mathbf{G}$ is primitive. It is also equivalent to show that $\mathbf{G}$ is strongly connected and there exists $(a,s_i) \in \mathbf{V}$ and $N \in \Nint$ such that every $n \ge N$ admits a walk $(a,s_i) \dhxrightarrow{n} (a,s_i)$. Since strong connectedness follows immediately, it is left to show the latter. Suppose $(a,s_i)$ is a pivot such that there exist $s_j \in S_k$ and walks $(a,s_i) \dhxrightarrow{N} (b_k,s_j)$ for every $b_k \in \alphabet$ as follows:
\[
\begin{matrix}
	(a,s_i) (a_{1,2},s_{l_{1,2}}) \cdots (a_{1,{N-1}},s_{l_{1,{N-1}}}) (a,s_j),\\
	(a,s_i) (a_{2,2},s_{l_{2,2}}) \cdots (a_{2,{N-1}},s_{l_{2,{N-1}}}) (b_2,s_j),\\
	\vdots\\
	(a,s_i) (a_{\norm{\alphabet},2},s_{l_{\norm{\alphabet},2}}) \cdots (a_{\norm{\alphabet},{N-1}},s_{l_{\norm{\alphabet},{N-1}}}) (b_{\norm{\alphabet}},s_j).
\end{matrix}
\]
Hence, the following are admissible words by $A$:
\[
\begin{matrix}
	a a_{1,2} \cdots a_{1,N-1} a,\\
	a a_{2,2} \cdots a_{2,N-1} b_2,\\
	\vdots\\
	a a_{\norm{\alphabet},2} \cdots a_{\norm{\alphabet},N-1} b_{\norm{\alphabet}}.
\end{matrix}
\]
From these, we are able to construct a word of length $n+1 \ge N+1$ with both starting and terminating symbol $a$. For instance, when $n=N+2$, we may observe $a_{1,N-2} a_{1,N-1} a= b_k a_{1,N-1} a$ for some $1 \le k \le \norm{\alphabet}$ and thus $a a_{k,2} \cdots a_{k,N-1} b_k a_{1,N-1} a$ is an admissible word by $A$. This process can be done for $N+1 \le n \le 2 N$, and further extension process for $n > 2N$ is done by a proper concatenation with the prefix $a a_{1,2} \cdots a_{1,N} a$. Now since $K$ is a primitive matrix, we can also prove that for every $s_i\in S_k$ and any sufficiently large $n \in \Nint$ there is an $(n+1)$-word admissible by $K$ which starts and terminates at $s_i$ simultaneously. Combining these two facts we are able to construct a walk $(a,s_i) \dhxrightarrow{n} (a,s_i)$ for all sufficiently large $n$, and the proof is completed. 
\end{proof}

Next, we show that the mixing property in the sense of a Markov tree shift results in the coincidence between the stem entropy and topological entropy. 

\begin{theorem} \label{thm:group_limit}
Let $X_{\mathbf{A}} \subseteq \mathcal{A}^G$ be a Markov tree shift on $G$. Suppose $\mathbf{G}=(\mathbf{V},\mathbf{E})$ is a graph representation of $X_{\mathbf{A}}$. Then the topological entropy $h = \lim_{n \to \infty} \frac{\log \pn{n}}{\norm{\Delta_n}}$ exists and $h = h^{(s)}$ provided $\mathbf{G}$ admits a pivot and is strongly connected.
\end{theorem}
\begin{proof}
        First, we show that $\liminf_{n \to \infty} \frac{\pnai{s_i}{n}{a}}{\norm{\bar{\Delta}^{(s_i)}_n}} = \liminf_{n \to \infty} \frac{\pnai{s_j}{n}{b}}{\norm{\bar{\Delta}^{(s_j)}_n}}$ for every $s_i, s_j \in S_k$ and $a, b \in \alphabet$. Suppose 
        \[\liminf_{n \to \infty} \frac{\log \pnai{s_i}{n}{a}}{\norm{\bar{\Delta}^{(s_i)}_n}}=\min \left\{\liminf_{n \to \infty} \frac{\log \pnai{s_l}{n}{c}}{\norm{\bar{\Delta}^{(s_l)}_n}}: s_l \in S_k,c \in \alphabet\right\}=:\underline{h}\] 
        and 
        \[\liminf_{n \to \infty} \frac{\log \pnai{s_j}{n}{b}}{\norm{\bar{\Delta}^{(s_j)}_n}}=\max \left\{\liminf_{n \to \infty} \frac{\log \pnai{s_l}{n}{c}}{\norm{\bar{\Delta}^{(s_l)}_n}}: s_l \in S_k, c \in \alphabet\right\}=:\overline{h}.
        \] 
        We show that $\overline{h}=\underline{h}$. To begin with, we gives an order on $G$ so that we are able to write $\{g_i\}_{i=1}^M = \{g \in \mctree: \norm{g}=N\}$ in the lexicographical order and introduce the notation
        \[
        \pnai{s_i}{N}{a;b_1,\cdots,b_{M}}:=\norm{\{u \in \alphabet^{\bar{\Delta}^{(s_i)}_N}: u \text{ is accepted by } t \in X_{\mathbf{A}}, u_{g_i}=b_i, \forall 1 \le i \le M\}}.
        \]
        Since $\mathbf{G}$ is strongly connected, there exists a walk $(a,s_i) \dhxrightarrow{n} (b,s_j)$ in $\mathbf{G}$. As a consequence, there exists $\pnai{s_i}{N}{a;b_1,\cdots,b,\cdots,b_{m}} \ge 1$ and thus 
\begin{align*}
    \pni{s_i}{N+n} &= \sum_{b_1,\cdots,b_M} \pnai{s_i}{N}{a;b_1,\cdots,b_{M}} \pnai{s_{l_1}}{n}{b_1} \pnai{s_{l_2}}{n}{b_2} \cdots \pnai{s_{l_M}}{n}{b_M} \\
    & \ge \pnai{s_i}{N}{a;b_1,\cdots,b,\cdots,b_{M}} \pnai{s_{l_1}}{n}{b_1} \pnai{s_{l_2}}{n}{b_2} \cdots \pnai{s_j}{n}{b} \cdots \pnai{s_{l_M}}{n}{b_M} \\
    & \ge \pnai{s_{l_1}}{n}{b_1} \pnai{s_{l_2}}{n}{b_2} \cdots \pnai{s_j}{n}{b} \cdots \pnai{s_{l_M}}{n}{b_M}.
\end{align*}
Hence, it yields
\begin{align*}
    \liminf_{n\to\infty} \frac{\log \pni{s_i}{N+n}}{\norm{\bar{\Delta}^{(s_i)}_{N+n}}} \ge \liminf_{n\to\infty} \frac{\log \pnai{s_{l_1}}{n}{b_1}}{\norm{\bar{\Delta}^{(s_{l_1})}_n}} \frac{\norm{\bar{\Delta}^{(s_{l_1})}_n}}{\norm{\bar{\Delta}^{(s_i)}_{N+n}}} + \cdots + \frac{\log \pnai{s_j}{n}{b}}{\norm{\bar{\Delta}^{(s_j)}_n}} \frac{\norm{\bar{\Delta}^{(s_j)}_n}}{\norm{\bar{\Delta}^{(s_i)}_{N+n}}} + \cdots + \frac{\log \pnai{s_{l_M}}{n}{b_M}}{\norm{\bar{\Delta}^{(s_{l_M})}_n}} \frac{\norm{\bar{\Delta}^{(s_{l_M})}_n}}{\norm{\bar{\Delta}^{(s_i)}_{N+n}}}.
\end{align*}
Note that $\lim_{n\to\infty}\frac{\norm{\bar{\Delta}^{(s_l)}_n}}{\norm{\bar{\Delta}^{(s_l)}_{N+n}}}$ is positive for every $s_l \in S_k$ and $\lim_{n\to\infty} \frac{\norm{\bar{\Delta}^{(s_{l_1})}_n}+\cdots+\norm{\bar{\Delta}^{(s_{l_M})}_n}}{\norm{\bar{\Delta}^{(s_{l_M})}_{N+n}}}=1$. It then follows that $\underline{h} \ge \overline{h}$.

Next, we show that $\liminf_{n \to \infty} \frac{\pnai{s_i}{n}{a}}{\norm{\bar{\Delta}^{(s_i)}}} = \limsup_{n \to \infty} \frac{\pnai{s_i}{n}{a}}{\norm{\bar{\Delta}^{(s_i)}_n}}$ for every $s_i \in S_k$ and $a \in \alphabet$. Suppose $(a,s_i)$ is a pivot in $\mathbf{G}$. Then, there exist $N \in \Nint$ and $s_j \in S_k$ such that every $(c,s_j) \in \mathbf{V}$ appears in one of the boundary patterns $b_1,\cdots, c, \cdots, b_M$, and is thus counted in $\pnai{s_i}{N}{a;b_1,\cdots,c, \cdots, b_M}$. On the other hand, it follows from the claim above that for every $\epsilon > 0$, there exists $N'$ such that for every $n \ge N'$,$s_l \in S_k$ and $c \in \alphabet$, 
\begin{equation} \label{eq:liminf_lb}
	\frac{\log \pnai{s_l}{n}{c}}{\norm{\bar{\Delta}^{(s_l)}_n}} \ge \overline{h}-\epsilon.
\end{equation} 
Hence, 
\begin{align*}
	\pnai{s_i}{N+n}{a} & = \sum_{b_1,\cdots,b_M} \pnai{s_i}{N}{a;b_1,\cdots,b_{M}} \pnai{s_{l_1}}{n}{b_1} \pnai{s_{l_2}}{n}{b_2} \cdots \pnai{s_{l_M}}{n}{b_M} \\
	& \ge \sum_{b_1,\cdots,b_M} \frac{1}{\norm{\alphabet}} \sum_{c:(c,s_j) \text{ appears in } b_1,\cdots,b_M} \pnai{s_i}{N}{a;b_1,\cdots,b_{M}} \pnai{s_{l_1}}{n}{b_1} \pnai{s_{l_2}}{n}{b_2} \cdots \pnai{s_{l_M}}{n}{b_M},
\end{align*}
for every product in the first line is counted no more than $\norm{\alphabet}$ times in the second summation in the second line. From equation \eqref{eq:liminf_lb}, one may further derive
\begin{align*}
	\pnai{s_i}{N+n}{a} & \ge \frac{1}{\norm{\alphabet}} \sum_{c} \sum_{\substack{b_1,\cdots,b_M:\\(c,s_j) \text{ appears in } b_1,\cdots,b_M}} \pnai{s_i}{N}{a;b_1,\cdots,b_{M}} \pnai{s_{l_1}}{n}{b_1} \pnai{s_{l_2}}{n}{b_2} \cdots \pnai{s_{l_M}}{n}{b_M}\\
	& \ge \frac{1}{\norm{\alphabet}} \sum_{c} \pnai{s_j}{n}{c} e^{(\overline{h}-\epsilon) (-\norm{\bar{\Delta}^{(s_j)}_n}+\sum_{s_{l_m}} \norm{\bar{\Delta}^{(s_{l_m})}_n}) } \\
	& = \frac{1}{\norm{\alphabet}} \pni{s_j}{n} e^{(\overline{h}-\epsilon) (-\norm{\bar{\Delta}^{(s_j)}_n}+\sum_{s_{l_m}} \norm{\bar{\Delta}^{(s_{l_m})}_n}) }.
	\end{align*}
	The inequality above yields that
	\begin{align*}
	& \hspace{1.5em} \liminf_{n\to\infty} \frac{\log \pnai{s_i}{N+n}{a}}{\norm{\bar{\Delta}^{(s_i)}_{N+n}}} \\
	& \ge \lim_{n\to\infty} \frac{\log \pni{s_j}{n}}{\norm{\bar{\Delta}^{(s_j)}_{n}}} \lim_{n\to\infty} \frac{\norm{\bar{\Delta}^{(s_j)}_{n}}}{\norm{\bar{\Delta}^{(s_i)}_{N+n}}}
	+ (\overline{h}-\epsilon) \lim_{n\to\infty} \frac{-\norm{\bar{\Delta}^{(s_j)}_n}+\sum_{s_{l_m}} \norm{\bar{\Delta}^{(s_{l_m})}_n}}{\norm{\bar{\Delta}^{(s_i)}_{N+n}}} \\
	& = h^{(s_j)} \cdot \lim_{n\to\infty} \frac{\norm{\bar{\Delta}^{(s_j)}_{n}}}{\norm{\bar{\Delta}^{(s_i)}_{N+n}}} + (\overline{h}-\epsilon) \cdot \lim_{n\to\infty} \frac{-\norm{\bar{\Delta}^{(s_j)}_n}+\sum_{s_{l_m}} \norm{\bar{\Delta}^{(s_{l_m})}_n}}{\norm{\bar{\Delta}^{(s_i)}_{N+n}}}.
\end{align*}
It follows as a result that $h^{(s_j)}=\overline{h} = h^{(s)}$, since $\lim_{n\to\infty} \frac{\sum_{s_{l_m}} \norm{\bar{\Delta}^{(s_{l_m})}_n}}{\norm{\bar{\Delta}^{(s_i)}_{N+n}}}=1$,  $\lim_{n\to\infty} \frac{\norm{\bar{\Delta}^{(s_j)}_{n}}}{\norm{\bar{\Delta}^{(s_i)}_{N+n}}} > 0$ and the righthand side of the inequality is a convex combination of $h^{(s_j)}$ and $\overline{h}-\epsilon$.

We are now ready to prove the proposition. Since it follows from \eqref{eq:top<=stem} that $\limsup_{n \to \infty} \frac{\log \pn{n}}{\norm{\Delta_n}} \le h^{(s)}$, it is left to show that $\liminf_{n \to \infty} \frac{\log \pn{n}}{\norm{\Delta_n}} \ge h^{(s)}$. For every $a \in \alphabet$, there exist $b_1, b_2, \ldots b_k \in \alphabet$ such that $\pna{n}{a} \ge \prod_{l=1}^k \pnai{s_1}{n-1}{b_l}$. It can be deduced from above that
\[
\liminf_{n \to \infty} \frac{\log \pn{n}}{\norm{\Delta_n}} \ge \liminf_{n \to \infty} \frac{\log \pna{n}{a}}{\norm{\Delta_n}} \ge \liminf_{n \to \infty} \sum_{l=1}^k \frac{\log \pnai{s_l}{n-1}{b_l}}{\norm{\bar{\Delta}^{(s_l)}_{n-1}}} \frac{\norm{\bar{\Delta}^{(s_l)}_{n-1}}}{\norm{\Delta_{n-1}}} = h^{(s)}.
\]
The proof is then finished.
\end{proof}

The corollary below follows immediately from Proposition \ref{prop:graph_and_hom_shift} and Theorem \ref{thm:group_limit}.

\begin{corollary}
    If $X_{\mathbf{A}}$ is a hom Markov tree shift, then the topological entropy $h$ exists and equals $h^{(s)}$ if $A$ is primitive.
\end{corollary}

Finally, we show that the assumption in Theorem \ref{thm:group_limit} is finitely checkable.

\begin{proposition} \label{prop:pivot_checkable}
	Let $X_{\mathbf{A}}$ be a Markov tree shift. Suppose $\mathbf{G}=(\mathbf{V},\mathbf{E})$ is a graph representation of $X_{\mathbf{A}}$. It is finitely checkable whether $\mathbf{G}$ admits a pivot and whether $\mathbf{G}$ is strongly connected.
\end{proposition}
\begin{proof}
	Since $\mathbf{G}$ is strongly connected if and only if the adjacency matrix $A_\mathbf{G}$ associated with $\mathbf{G}$ is irreducible, it is clearly finitely checkable. To see the admittance of pivot is also finitely checkable, we define the matrix $A_n$ for all $n \in \Zint_+$ as follows:
	\[
	A_n((a,s_i),(b,s_j))=\begin{cases}
		1 & \text{if } (A_\mathbf{G})^n((a,s_i),(b,s_j))=1, \\
		0 & \text{otherwise.}
	\end{cases}
	\]
	It is then clear that $\mathbf{G}$ admits a pivot if and only if there exist $s_i,s_j \in S_k$, $a \in \alphabet$, and $n \in \Zint_+$ such that $A_n((a,s_i),(b,s_j))=1$ for all $b \in \alphabet$. Since $\norm{\{A_n: n \ge 0\}} \le 2^{\norm{\mathbf{V}}^2}$ and $A_n$ is eventually periodic, there exist $0 \le N_1 \le N_2 \le 2^{\norm{\mathbf{V}}^2}$ such that $A_{N_1+n}=A_{N_2+n}$ for all $n \ge 0$. In other words, $\mathbf{G}$ admits a pivot if and only if there exist $s_i,s_j \in S_k$, $a \in \alphabet$, and $1 \le n \le 2^{\norm{\mathbf{V}}^2}$ such that $A_n((a,s_i),(b,s_j))=1$ for all $b \in \alphabet$. This implies that admittance of a pivot is finitely checkable.
\end{proof}

\begin{appendix}
	\section{An Attempt toward the Existence of Topological Entropy}
	
	This section presents an attempt toward the existence of topological entropy by exploiting the composition of colors on the boundary of all $n$-blocks. Suppose $X_{\mathbf{A}}$ is given. We denote by a vector $\mathbf{v} \in \Zint_+^{\norm{\alphabet} \norm{S_k}}$ the product $\prod_{(a,s_i)} (\pnaip{s_i}{n}{a})^{\mathbf{v}_{(a,s_i)}}$. Note that
	$$
	W:=\{\sum_{\mathbf{v} \in \Zint_+^{\norm{\alphabet} \norm{S_k}}} r_{\mathbf{v}} \cdot \mathbf{v}: r_{\mathbf{v}} \in \Zint, r_{\mathbf{v}} \ne 0 \text{ for finitely many } \mathbf{v} \in \Zint_+^{\norm{\alphabet} \norm{S_k}}\}
	$$
	is a vector space with a basis $\Zint_+^{\norm{\alphabet} \norm{S_k}}$. Define the linear transformation $F:W \to W$ as
	\[
	(F(\mathbf{v}))_{(a,s_i)}=\begin{cases}
		1, & \text{if } \mathbf{v}_{(a,s_i)} > 0;\\
		0, & \text{if } \mathbf{v}_{(a,s_i)} = 0,
	\end{cases}
	\]
	and the simplified representation $F^\ast(\mathbf{v})$ of $F(\mathbf{v})$ as
	\[
	F^\ast(\sum_{\mathbf{v}} r_{\mathbf{v}} \cdot \mathbf{v})=\sum_{\mathbf{v}} r'_{\mathbf{v}} \cdot F(\mathbf{v}),
	\]
	where
	\[
	r'_\mathbf{v}=\begin{cases}
		1, & \text{if } r_\mathbf{v} > 0;\\
		0, & \text{if } r_\mathbf{v} = 0.
	\end{cases}
	\]
	Define the shift transformation $\sigma:W \to W$ by 
	\[\sigma(\mathbf{v})=\sigma\left(\prod_{(a,s_i)} {\pnaip{s_i}{n}{a}}^{\mathbf{v}_{(a,s_i)}}\right)=\prod_{(a,s_i)} \sum_{b} \prod_{j:K(i,j)=1} A_i(a,b) \pnaip{s_j}{b}{n}\]
	Suppose $x, y \in W$. We denote $x \succeq y$ if every term $\mathbf{v}$ appearing in $F^\ast(x)$ admits a term $\mathbf{w}$ appearing in $F^\ast(y)$ satisfying $\mathbf{v}_{(a,s_i)} \ge \mathbf{w}_{(a,s_i)}$ for every $a \in \alphabet$ and every $s_i \in S_k$.
	
	\begin{proposition}
		$\lim_{n \to \infty} \frac{\log \pn{n}}{\norm{\Delta_n}}$ exists and equals $h^{(s)}$ if there exist $N_1, N_2 \in \Nint$ and $s_i \in S_k$ such that $\sigma^{N_1}(\pni{s_i}{n}) \succeq \sigma^{N_2}(\pn{n})$
	\end{proposition}
	\begin{proof}
		Denote $M = \norm{\{g \in G: \norm{g} = N_1\}}$, $x=\sigma^{N_1}(\pni{s_i}{n})$ and $y=\sigma^{N_2}(\pn{n})$. Since $x \succeq y$, every term $\mathbf{v}$ appearing in $F^\ast(x)$ admits a term $\phi(\mathbf{v})$ appearing in $F^\ast(y)$ satisfying $\mathbf{v}_{(a,s_i)} \ge \phi(\mathbf{v})_{(a,s_i)}$ for every $a \in \alphabet$ and every $s_i \in S_k$. In this proof, we denote $[n,\mathbf{v}]$ for every $\mathbf{v}=\prod_{(a,s_i)} (\pnaip{s_i}{n}{a})^{\mathbf{v}_{(a,s_i)}} \in W$ for an emphasis on the size of the block. 
		
		Note that since $\lim_{n \to \infty} \frac{\log \pni{s_i}{n}}{\norm{\bar{\Delta}^{(s_i)}}} = h^{(s)}$, there exists 
		\[
		[n, \mathbf{v}_n]=(\pnaip{s_{l_1}}{n}{a_{n;1}})^{\mathbf{v}_{(a_{n;1},s_{l_1})}} \cdot \cdots \cdot (\pnaip{s_{l_M}}{n}{a_{n;M}})^{\mathbf{v}_{(a_{n;M},s_{l_M})}}
		\] appearing in $x$ such that $\lim_{n \to \infty} \frac{\log[n, \mathbf{v}_n]}{\norm{\bar{\Delta}^{(s_i)}_n}}=h^{(s)}$. Hence, for every $\epsilon > 0$, there exists $N \in \Nint$ such that 
		\[\frac{\log[n, \mathbf{v}_n]}{\norm{\bar{\Delta}^{(s_i)}_n}} > h^{(s)}-\epsilon\] 
		and that \[\frac{\log \pnaip{s_j}{n}{b}}{\norm{\bar{\Delta}^{(s_j)}_n}} < h^{(s)}-\epsilon\]
		for every $n \ge N$, $b \in \alphabet$ and $s_j \in S_k$. Hence,
		\[
		\frac{\log \pnaip{s_{l_m}}{n}{a_{n;m}}}{\norm{\bar{\Delta}^{(s_j)}_n}} \frac{(\mathbf{v}_n)_{(a_m,s_{l_m})}}{M} \ge (h^{(s)}-\epsilon) - (h^{(s)}+\epsilon) \frac{M-(\mathbf{v}_n)_{(a_m,s_{l_m})}}{M},
		\]
		and thus \[
		\frac{\log \pnaip{s_{l_m}}{n}{a_m}}{\norm{\bar{\Delta}^{(s_j)}_n}} \ge h^{(s)}-2 M \epsilon.
		\]
		Now observe that
		\[
		\frac{\log \pn{n}}{\norm{\Delta_n}} \ge \frac{\phi([n,\mathbf{v}_n])}{\norm{\Delta_n}} \ge h^{(s)}-2 M \epsilon,
		\]
		for all $n \ge N$. The proof is thus finished.
	\end{proof}

	\section{Computation of Stem Entropy}
		In this section, we provide the pseudo codes for \textbf{1.}~computation for topological entropy of Markov tree shift on the Cayley graph and \textbf{2.}~computation for stem entropy of Markov tree shift, shown in Algorithm \ref{alg:topological_entropy} and Algorithm \ref{alg:stem_entropy} respectively. In the following, we denote by $\odot$ the entrywise product of vectors.
	
	\begin{algorithm} \label{alg:topological_entropy}
		\caption{Topological entropy of hom Markov tree shift on the Cayley graph}
		\SetKwFunction{nte}{normalized\_tree\_entropy}
		\SetKwInOut{Input}{input}
		\SetKwInOut{Output}{output}
		\SetKwProg{Fn}{Function}{}{end}
		\SetKw{Break}{break}
		
		\Input{
			\vspace{-0.2em} $\mathbf{A}=(A_1,A_2,\cdots, A_d)$: $d$ binary matrices of dimension $k$. \\
			$iter$: maximum of iterations in execution \\
			$\epsilon$: threshold for convergence
		}
		
		\Output{$h$: approximation of entropy, where $h_n:=\frac{\log \max_a \pna{n}{a}}{\lvert \Delta_n \rvert}$.}
		\BlankLine
		
		\Fn{\nte{$\mathbf{A}$,iter,$\epsilon$}}{
			$\bar{\mathbf{p}}_0=\left(\pnab{0}{1},\pnab{0}{2},\cdots,\pnab{0}{k}\right)^t \leftarrow \left(1,1,\cdots,1\right)^t$\;
			$r_0 \leftarrow 1$\;
			$t_0 \leftarrow \log r_0$\;
			$h_0 \leftarrow t_0 / \lvert \Delta_0 \rvert$\;
			\For{$n \in \{1, 2, \cdots, iter-1\}$}{
			    $\bar{\mathbf{p}}_n=(\pnab{n}{1},\pnab{n}{2},\cdots,\pnab{n}{k})^t \leftarrow (A_1 \bar{\mathbf{p}}_{n-1}) \odot \cdots \odot (A_d \bar{\mathbf{p}}_{n-1})$\;
				$r_n \leftarrow \max_a \pnab{n}{a}$\;
				$\bar{\mathbf{p}}_n \leftarrow \bar{\mathbf{p}}_n / r_n$\;
				$t_n \leftarrow d \cdot t_{n-1} + \log r_n$\;
				$h_n \leftarrow t_n / \norm{\Delta_n}$\;
				\If{$\norm{h_n-h_{n-1}} < h_{n-1} \cdot \epsilon$ or $h_{n} < \epsilon$}{
					\Break\;
				}
			}
			\Return $h$
		}
	\end{algorithm}
	\begin{remark}
		The idea behind Algorithm \ref{alg:topological_entropy} is given as follows. Suppose 
		\[
		\begin{cases}
		    \mathbf{p}_n=f(\mathbf{p}_{n-1}):=(A_1 \mathbf{p}_{n-1}) \odot \cdots \odot (A_d \mathbf{p}_{n-1}) \\
		    \mathbf{p}_{0}=\left(1,1,\cdots,1\right)^t.
		\end{cases}
		\]
		It is shown (see for example \cite{BC-N2017}) that the above system is exactly the vector of the number of blocks:
		\[
		\mathbf{p}_n=(\pna{n}{1},\pna{n}{1}, \cdots, \pna{n}{k}).
		\]
		Let $\{r_n > 0: n \ge 0\}$ be a given sequence of positive real numbers. Define the normalized system as 
		\[
		\begin{cases}
		    \bar{\mathbf{p}}_n=g(\bar{\mathbf{p}}_{n-1}):=\frac{f(\bar{\mathbf{p}}_{n-1})}{r_n},\\
		    \bar{\mathbf{p}}_{0}=(1/r_0,1/r_0,\cdots,1/r_0)^t
		\end{cases}
		\]
		It is noteworthy that the following equality holds:
		\begin{align*}
			\mathbf{p}_n&=f^n(\mathbf{p}_0)\\
			&=g^n(\bar{\mathbf{p}}_0) \cdot r(0)^{d^n} r(1)^{d^{n-1}} \cdots r(n)^{d^0} \\
			&=\bar{\mathbf{p}}_n \cdot r(0)^{d^n} r(1)^{d^{n-1}} \cdots r(i)^{d^0}
		\end{align*}
		Since $r_n$ is chosen to be the maximal element in $\mathbf{p}_n$ in Algorithm \ref{alg:topological_entropy}, the maximal element in $\bar{\mathbf{p}}(i)$ is 1 and thus 
		\begin{equation} \label{eq:relation1}
			t_n=\log \max_a \pna{n}{a} = d^n \log r_0 + d^{n-1} \log r_1 + \cdots + d^{0} \log r_n, 
		\end{equation}
		and
		\[
		h_n=\frac{\log \max_a \pna{n}{a}}{\lvert \Delta_n \rvert}.
		\] 
		In fact, if $r_n$ is defined as in the algorithm, then $r_n$ is a rational number and
		\begin{align}
			h(X_\mathbf{A})= \lim_{n \to \infty} \frac{\log \max_a \pna{n}{a}}{d^{n+1}/{d-1}}=\sum_{n=0}^{\infty} \log r_n \cdot \frac{d-1}{d^{n+1}}.
		\end{align}
		In particular, if $X_{(A,A,\cdots,A)}$ is a hom Markov tree shift with $A$ an essential matrix , i.e., for every $b \in \alphabet$ there exists $b \in \alphabet$ satisfying $A(a,b)=1$, then $\norm{\alphabet}^d \ge r_n \ge 1$ for all $n \ge 0$ and
		\[
		\sum_{n=0}^{N} \log r_n \cdot \frac{d-1}{d^{n+1}} \le h(X_\mathbf{A}) \le \sum_{n=0}^{N} \log r_n \cdot \frac{d-1}{d^{n+1}} + \sum_{n=N+1}^{\infty} d \log \norm{\alphabet} \cdot \frac{d-1}{d^{n+1}}.
		\]
	\end{remark}
	
	\begin{algorithm} \label{alg:stem_entropy}
		\caption{Stem entropy of Markov tree shift}
		\SetKwFunction{nte}{normalized\_mctree\_entropy}
		\SetKwInOut{Input}{input}
		\SetKwInOut{Output}{output}
		\SetKwProg{Fn}{Function}{}{end}
		
		\Input{
			\vspace{-0.2em} $K$: binary matrix of dimension $k$. \\
			$\mathbf{A}=(A_1,A_2,\cdots, A_d)$: $d$ binary matrices of dimension $k$. \\
			$iter$: maximum of iterations in execution \\
			$\epsilon$: threshold for convergence
		}
		
		\Output{$h^{(s_j)}$: approximation of entropy, where $h^{(s_j)}_n:=\frac{\log \max_a \pnai{s_j}{n}{a}}{\lvert \Delta_n \rvert}$.}
		\BlankLine
		
		\Fn{\nte{$\mathbf{A}$,iter,$\epsilon$}}{
			\For{$j \in \{1, 2, \cdots, d\}$}{
				$\bar{\mathbf{p}}^{(s_j)}_0=\left(\pnaib{s_j}{0}{1},\pnaib{s_j}{0}{2},\cdots,\pnaib{s_j}{0}{k}\right)^t \leftarrow \left(1,1,\cdots,1\right)^t$\;
				$r^{(s_j)}_n \leftarrow \max_a \pnaib{s_j}{n}{a}$\;
				$t^{(s_j)}_0 \leftarrow \log r^{(s_j)}_0$\;
				$h^{(s_j)}_0 \leftarrow s^{(s_j)}_0 / \norm{\bar{\Delta}^{(s_j)}_0}$\;
			}
			\For{$n \in \{1, 2, \cdots, iter-1\}$}{
				\For{$j \in \{1, 2, \cdots, d\}$}{
				    $\bar{\mathbf{p}}^{(s_j)}_n=\left(\pnaib{s_j}{n}{1},\pnaib{s_j}{n}{2},\cdots,\pnaib{s_j}{n}{k}\right)^t \leftarrow (A_1 \mathbf{p}^{(s_1)}_{n-1})^{K(s_j,s_1)} \odot \cdots \odot (A_d \mathbf{p}^{(s_d)}_{n-1})^{K(s_j,s_d)}$\;
					$r^{(s_j)}_n \leftarrow 1$\;
					$\bar{\mathbf{p}}^{(s_j)}_n=\left(\pnaib{s_j}{n}{1},\pnaib{s_j}{n}{2},\cdots,\pnaib{s_j}{n}{k}\right)^t \leftarrow \bar{\mathbf{p}}^{(s_j)}_n / r^{(s_j)}_n$\;
					$t^{(s_j)}_n \leftarrow K(s_j,:) \cdot \left(t^{(s_1)}_{n-1},\cdots,t^{(s_d)}_{n-1}\right) + \log r^{(s_j)}_n$\;
					$h^{(s_j)}_n \leftarrow t^{(s_j)}_n / \norm{\bar{\Delta}^{(s_j)}_n}$\;
				}
				\If{$\sum_{j=1}^d \norm{h^{(s_j)}_n-h^{(s_j)}_{n-1}} < \sum_{j=1}^d h^{(s_j)}_{n-1} \cdot \epsilon$ or $h^{(s_j)}_n < \epsilon$}{
					break\;
				}
			}
			\Return $\left(h^{(s_1)};h^{(s_d)};\cdots;h^{(s_d)}\right)$
		}
	\end{algorithm}
	\begin{remark}
		As an analogy of Algorithm \ref{alg:topological_entropy}, the numbers of blocks satisfy the following recursive system.
		\[
		\begin{cases}
		\mathbf{p}^{(s_j)}_n=(A_1 \mathbf{p}^{(s_1)}_{n-1})^{K(s_j,s_1)} \odot \cdots \odot (A_d \mathbf{p}^{(s_d)}_{n-1})^{K(s_j,s_d)} \\
		\mathbf{p}^{(s_j)}_0=(1,1,\cdots,1)^t.
		\end{cases}
		\]
		In the same manner, given any positive sequence of $\{r^{(s_j)}_n: n \ge 0, s_j \in G\}$, one may define the normalized system as 
	    \[
		\begin{cases}
		\bar{\mathbf{p}}^{(s_j)}_n=(A_1 \bar{\mathbf{p}}^{(s_1)}_{n-1})^{K(s_j,s_1)} \odot \cdots \odot (A_d \bar{\mathbf{p}}^{(s_d)}_{n-1})^{K(s_j,s_d)} / r^{(s_j)}_n \\
		\bar{\mathbf{p}}^{(s_j)}_0=(1/r^{(s_j)}_0,1/r^{(s_j)}_0,\cdots,1/r^{(s_j)}_0)^t.
		\end{cases}
		\]
		Denote by $\mathbf{f}=(f_1, f_2, \cdots, f_d)$ and $\mathbf{g}=(g_1, g_2, \cdots, g_d)$ the map $\mathbf{p}_{n-1} \stackrel{\mathbf{f}}{\rightarrow} \mathbf{p}_{n}$ and the map $\bar{\mathbf{p}}_{n-1} \stackrel{\mathbf{g}}{\rightarrow} \bar{\mathbf{p}}_{n}$, respectively. Similar to the above,
		\begin{align}
			\mathbf{p}^{(s_j)}_n&=f_j(\mathbf{f}^{n-1}(\mathbf{p}^{(s_1)}_0,\mathbf{p}^{(s_2)}_0,\cdots,\mathbf{p}^{(s_d)}_0))\\
			&=g_j(\mathbf{g}^{n-1}(\bar{\mathbf{p}}^{(s_1)}_0,\bar{\mathbf{p}}^{(s_2)}_0,\cdots,\bar{\mathbf{p}}^{(s_d)}_0)) \cdot \exp(t^{(s_j)}_{n-1}) \\
			&=\bar{\mathbf{p}}^{(s_j)}_n \cdot \exp(t^{(s_j)}_n),
		\end{align}
		where $\max \bar{\mathbf{p}}^{(s_j)}_n = 1$ and
		\begin{equation} \label{eq:relation2}
			t^{(s_j)}_n=\log \max_a \pnai{s_j}{n}{a}.
		\end{equation}
		Hence
		\[
		    h^{(s_j)}_n=\frac{\log \max_a \pnai{s_j}{n}{a}}{\norm{\bar{\Delta}^{(s_j)}_n}},
		\] 
		which tends to the stem entropy of $X_{\mathbf{A}}$.
	\end{remark}
	The experiments are done with mpmath library of python under the following configuration: the precision digits for floating-point number dps=5000, the threshold $\epsilon=1E-50$, and the relations $K=(1,1,0,1;1,1,1,0;0,1,1,1;1,0,1,1)$.

    \begin{table}[]
        \centering
        \begin{tabular}{c|c|c|c|c}
        	$A_1$ & $A_2$ & stem entropy & topological entropy & iteration \\
        	\hline
        	$[0,1;1,1]$ & $[1,1;1,0]$ & 0.1261881372008 & 0.1261881372008 & 37\\
        	$[1,1;1,0]$ & $[1,1;1,0]$ & 0.2332621211030 & 0.2332621211030 & 34\\
        	$[0,1,0;1,0,1;0,1,0]$ & $[0,1,1;1,0,0;0,1,1]$ & 0.1681464340595 & 0.1681464340595 & 36
        \end{tabular}
        \caption{Numerical experiments on the stem entropy of $X_{A_1,A_2,A_1^t,A_2^t}$ over the free group ($\log$ is computed with base 10.)}
        \label{tab:my_label}
    \end{table}
    
    \begin{table}[]
        \centering
        \begin{tabular}{c|c|c|c|c}
        	$A_1$ & $A_2$ & stem entropy & topological entropy & iteration \\
        	\hline
        	$[0,1;1,1]$ & $[1,1;1,0]$ & 0.1261881372008 & 0.1261881372008 & 37\\
        	$[1,1;1,0]$ & $[1,1;1,0]$ & 0.2332621211030 & 0.2332621211030 & 34\\
        	$[0,1,0;1,0,1;0,1,0]$ & $[0,1,1;1,0,0;0,1,1]$ & 0.1681464340595 & 0.1681464340595 & 36
        \end{tabular}
        \caption{Numerical experiments on the stem entropy of $X_{A_1,A_2,A_1^t,A_2^t}$ over the free group ($\log$ is computed with base 10.)}
        \label{tab:free_group}
    \end{table}
    \begin{table}[]
        \centering
        \begin{tabular}{c|c|c|c|c}
        	$A_1$ & $A_2$ & stem entropy & topological entropy & iteration \\
        	\hline
        	$[1,1;1,0]$ & $[1,1;1,0]$ & 0.2178219813166 & 0.2178219813166 & 82\\
        	$[0,1;1,1]$ & $[0,1;1,1]$ & 0.2178219813166 & 0.2178219813166 & 82\\
        	$[0,1;1,1]$ & $[1,1;1,0]$ & 0.1267559612313 & 0.1267559612313 & 73\\
        	$[1,1;1,0]$ & $[0,1;1,1]$ & 0.1267559612313 & 0.1267559612313 & 73\\
        \end{tabular}
        \caption{Numerical experiments on the stem entropy of $X_{A_1,A_2}$ over Fibonacci-Cayley tree generate by $S_k$, where $K=[1,1;1,0]$ ($\log$ is computed with base 10.)}
        \label{tab:fibonacci_tree}
    \end{table}
\end{appendix}

\section*{Acknowledgment}
We are appreciated for the comments from the anonymous referees, which greatly improve the readability of the article.

\bibliographystyle{amsplain_abbrv}
\bibliography{grece}

\end{document}